\documentclass[12pt]{amsart}

\newcommand{\cal}[1]{\mathcal{#1}}
\usepackage{john}
\usepackage{hyperref}

\numberwithin{theorem}{section}

\newcommand{\Ref}{\oper{Ref}}

\renewcommand{\cris}{\crys}

\usepackage{float}
\restylefloat{table}

\usepackage{setspace}
\title{Smoothness of definite unitary eigenvarieties at critical points}
\author{John Bergdall}
\date {\today}
\address{John Bergdall\\Department of Mathematics and Statistics \\ Boston University \\ 111 Cummington Mall \\ Boston, MA 02215\\USA}
\email{bergdall@math.bu.edu}
\urladdr{http://math.bu.edu/people/bergdall}

\subjclass[2000]{11F33 (11F80, 11F55)}

\usepackage{fullpage}

\usepackage[all,cmtip]{xy}

\begin{document}
\begin{abstract}
We compute an upper bound for the dimension of the tangent spaces at classical points of certain eigenvarieties associated with definite unitary groups, especially including the so-called critically refined cases. Our bound is given in terms of ``critical types'' and when our bound is minimized it matches the dimension of the eigenvariety. In those cases, which we explicitly determine, the eigenvariety is necessarily smooth and our proof also shows that the completed local ring on the eigenvariety is naturally a certain universal Galois deformation ring.
\end{abstract}
\maketitle 
\setcounter{tocdepth}{1}

\section{Introduction}
Let $p$ be a prime number. Eigenvarieties, referring to $p$-adic families of automorphic forms, first appeared thirty years ago in the work of Hida on $p$-ordinary elliptic cuspidal eigenforms. Coleman and Mazur, and Buzzard, removed the ordinary condition in the final decade of the 20th century, constructing $p$-adic families passing through all finite slope cuspidal eigenforms. A number of authors have since given  constructions of eigenvarieties in a wide range of situations.  Despite the growing interest in the subject, basic geometric properties of eigenvarieties remain occluded.\footnote{For example: it is not known if the tame level 1, $2$-adic eigencurve has finitely many connected components.}

The goal of this article is to give upper bounds for the dimensions of the tangent spaces at classical points on eigenvarieties associated with definite unitary groups. These bounds are given in terms of {\em critical types} of triangulations of local Galois representations at the $p$-adic places. The more critical a point is, the larger the upper bound.  

We also give an exact condition for when our bound is minimized. In that case, our bound equals the dimension of the corresponding eigenvariety and we get a smoothness statement as well. For the rest of this introduction we set notation and state the main theorem.

\subsection{Critical types of $p$-refined automorphic representations}
Let $F/F^+$ be a CM extension of number fields such that each $p$-adic place of $F^+$ splits in $F$. We will use $\bf G$ to denote a rank $n$ unitary group associated to this extension. We assume $\bf G\times_{F^+} F \simeq {\GL_n}_{/F}$ and that $\bf G$ is compact at infinity. In the body of the text, we will also make further technical assumptions that we omit now (see Section \ref{section:application}). If $v$ is a $p$-adic place of $F^+$ then the choice of a place $\twid v \dvd v$ in $F$ defines an isomorphism $\bf G(F_{v}^+) \simeq \bf \GL_n(F_{\twid v}) = \GL_n(F_v^+)$. Thus, the local components of automorphic representations for $\bf G$ at $p$-adic places are irreducible smooth representations of $\GL_n(F_v^+)$.

Associated with an automorphic representation $\pi$ for $\bf G$ is a $p$-adic Galois representation
\begin{equation*}
\rho_{\pi} : G_{F} \goto \GL_n(\bar \Q_p).
\end{equation*}
The representation $\rho_{\pi}$ is geometric in the sense of Fontaine and Mazur: $\rho_\pi$ is unramified except at finitely many places and it is de Rham at the $p$-adic places.

We now fix an automorphic representation $\pi$ of $\bf G$ lying in the unramified principal series at the $p$-adic places $v \dvd p$ in $F^+$. For such $v$, we also fix the choice of $\twid v \dvd v$ in $F$ and denote by $\rho_{\pi,\twid v}$ its restriction to a decomposition group at $\twid v$. Since the local component $\pi_v$ at a place $v \dvd p$ is an unramified principal series, the representation $\rho_{\pi,\twid v}$ is crystalline. Note that $F_{\twid v} = F_v^+$ and, to emphasize that the choice of $\twid v$ does not matter substantially, we write $F_v^+$ for this $p$-adic field. Finally, we choose a finite extension $L/\Q_p$, inside a fixed algebraic closure $\bar \Q_p$, such that the image of $\rho_{\pi}$ is contained in $\GL_n(L)$, and we assume that $L$ contains the image of any embedding $\tau: F_v^+ \inject \bar \Q_p$ for each $v$.

Along with $\pi$ and the distinguished places $\twid v$, we also fix a collection of triangulations $(P_{\twid v,\bullet})_{\twid v}$. Specifically, we consider the $(\varphi,\Gamma_{F_{v}^+})$-module $D_{\pi,\twid v} := D_{\rig}^{\dagger}(\rho_{\pi,\twid v})$ over the Robba ring $\cal R_{F_v^+,L}$ (see Section \ref{section:reminder}), and the triangulation $P_{\twid v,\bullet}$ is a full filtration
\begin{equation*}
P_{\twid v,\bullet}: \;\; (0) = P_{\twid v,0} \sci P_{\twid v,1} \sci \dotsb \sci P_{\twid v,n-1} \sci P_{\twid v, n} = D_{\pi,\twid v}
\end{equation*} 
of saturated $(\varphi,\Gamma_{F_v^+})$-submodules. We refer to $\pi$ together with the choice of triangulations as a $p$-refined automorphic representation (the terminology goes back to \cite{Mazur-padicvariation}).

The graded pieces of the triangulation $P_{\twid v,\bullet}$ have rank one and so $P_{\twid v,\bullet}$ defines an ordered tuple $(\delta_{\twid v,1},\dotsc,\delta_{\twid v,n})$ of continuous characters $\delta_{\twid v,i} : (F_v^+)^\times \goto L^\times$ as in \cite[Section 6.2]{KedlayaPottharstXiao-Finiteness} (see also Section \ref{subsec:first-subsec}). For each embedding $\tau:F_v^+\inject L$ there exists an integer $s_{i,\twid v,\tau}$ (the $\tau$-Hodge--Tate weight of $\delta_{\twid v,i}$) such that $\delta_{\twid v,i}(z) = \prod_{\tau} \tau(z)^{-s_{i,\twid v,\tau}}$ for all $z$ in $\cal O_{F_v^+}^\times$. (Our normalization gives the identity character Hodge--Tate weight $-1$ at each embedding $\tau$.)\label{pageref:weight}

On the other hand, for each embedding $\tau: F_v^+ \inject L$, we also have the (necessarily distinct) Hodge--Tate weights $h_{1,\twid v,\tau} < h_{2,\twid v,\tau} < \dotsb < h_{n,\twid v,\tau}$ of $D_{\pi,\twid v}$. For  fixed $\twid v$ and $\tau$, the two sets $\set{s_{i,\twid v,\tau} \st i = 1,\dotsc, n}$ and $\set{h_{i,\twid v, \tau} \st i = 1,\dotsc, n}$ are equal, so we denote by $\sigma_{\twid v,\tau}$ the unique permutation defined by $s_{i,\twid v,\tau} = h_{\sigma_{\twid v,\tau}(i), \twid v, \tau}$.\footnote{It may make sense to replace $\sigma_{\twid v,\tau}$ by its inverse (for formulae in representation theory to work out cleaner). The results below (Theorems \ref{theorem:intro-theorem} and \ref{theorem:intro-local-theorem}) do not depend on this choice.}
\begin{definition}
The critical type of the triangulation $P_{\twid v,\bullet}$ of $D_{\pi,\twid v}$ is the collection of permutations $(\sigma_{\twid v,\tau})_{\tau}$. The triangulation $P_{\twid v,\bullet}$ is called non-critical if $\sigma_{\twid v,\tau} = \id$ for each $\tau$.
\end{definition}
The non-critical case is the most common. For instance, a $p$-refined automorphic representation of non-critical slope is non-critical (see \cite[Remark 2.4.6(ii)]{BellaicheChenevier-Book}) and having non-critical slope is generic on an eigenvariety. Nevertheless, interesting arithmetic phenomena occur in critical situations (see \cite[Theorem 2]{Bellaiche-CriticalpadicLfunctions} for example) and it seems less difficult for a triangulation to be critical as $n \goto \infty$. For contrast, if $n=2$ and $F_v^+ = \Q_p$ then a critical triangulation at $\twid v$ exists if and only if $\rho_{\pi,\twid v}$ is abelian.

\subsection{Main result}\label{subsec:main-result}
An eigenvariety $p$-adically interpolates $p$-refined automorphic representations $(\pi, (P_{\twid v,\bullet}))$. We refer to \cite{Chenevier-pAdicAutomorphicForm,Emerton-InterpolationEigenvariety} and Section \ref{section:application} for details. Here, we fix a {\em minimal} eigenvariety $X$ containing the pair $x:=(\pi,(P_{\twid v,\bullet}))$. Thus, $X$ is a rigid analytic space over $\Q_p$, equidimensional of dimension $(F^+:\Q)\cdot n$. It implicitly depends on the choice of a tame level; second, the minimal condition essentially means that the point $x$ is not lying at the intersection of two eigenvarieties obtained from smaller tame levels. Our main theorem is a bound on the dimension of the Zariski tangent space $T_{X,x}$ of $X$ at $x$ in terms of critical types.  To state it, we need two notations.

The representation $\rho_{\pi}$ is conjugate self-dual, up to a twist, and if $\rho_{\pi}$ is absolutely irreducible then the natural action of $G_{F}$ on the adjoint representation $\ad \rho_{\pi}$ extends to an action of the absolute Galois group $G_{F^+}$ (see Section \ref{subsec:ub}). We  denote by $H^1_f(G_{F^+}, \ad \rho_{\pi})$ the corresponding Bloch--Kato Selmer group \cite{BlochKato-TamagawaNumbersOfMotives}.

If $\sigma$ is a permutation of $\set{1,\dotsc,n}$, we let $\ell(\sigma) = \set{(i,j) \st i < j \text{ and $\sigma(i) > \sigma(j)$}}$. This is also the length of a minimal expression of $\sigma$ as a product of simple transpositions. (A simple transposition is a transposition interchanging two consecutive integers $i$ and $i+1$.) We also write $c(\sigma)$ for the number of orbits of the group generated by $\sigma$ acting on $\set{1,\dotsc,n}$. For example, $\ell(\id) = 0$ and $c(\id) = n$.

In the next theorem, we refer to Definition \ref{defi:regular-generic} for the notion of a regular generic triangulation. For an idea, ``regular'' is a simplicity condition on crystalline eigenvalues.

\begin{theorem}[Theorem \ref{theorem:main-theorem-text}]\label{theorem:intro-theorem}
Suppose that $\pi$ is an automorphic representation of $\G$ which is unramified at each $p$-adic place and $(P_{\twid v,\bullet})$ is a collection of regular generic  triangulations for $\rho_{\pi}$ at the $p$-adic places. Assume that $\rho_{\pi}$ is irreducible and $X$ is the minimal eigenvariety containing the point $x = (\pi,(P_{\twid v,\bullet}))$. Then, 
\begin{equation*}
\dim T_{X,x} \leq \dim H^1_f(G_{F^+}, \ad \rho_{\pi}) + \sum_{\twid v,\tau} \ell(\sigma_{\twid v,\tau}) + c(\sigma_{\twid v,\tau}),
\end{equation*}
where $(\sigma_{\twid v,\tau})_{\twid v,\tau}$ are the critical types of the triangulations at $x$.
\end{theorem}
 We also have the following direct corollary.
\begin{corollary}[Corollary \ref{corollary:final-corollary}]\label{corollary:intro-corollary}
With the notation and assumptions of Theorem \ref{theorem:intro-theorem}, also assume that $H^1_f(G_{F^+}, \ad \rho_{\pi}) = (0)$ and each $\sigma_{\twid v,\tau}$ is a product of distinct simple transpositions. Then, $X$ is smooth at the point $x$.
\end{corollary}
The Selmer group is conjectured to vanish;\footnote{In fact, the larger space $H^1_f(G_F,\ad \rho_{\pi})$ is conjectured to vanish.} it has been proven in many cases, under hypotheses inherent to the methods of Taylors--Wiles and Kisin (see \cite{Allen-SelmerGroups} and \cite[Section 4]{BreuilHellmanSchraen-TWEigen}). Denote by $\bar \rho_{\pi}^{\oper{ss}}$ the semi-simplification of any mod $p$ reduction of $\rho_{\pi}$. As a concrete version of Corollary \ref{corollary:intro-corollary}, the main theorem of \cite{Allen-SelmerGroups} implies:
\begin{corollary}
With the notation and assumptions of Theorem 1.2, assume that each $\sigma_{\twid v,\tau}$ is a product of distinct simple transpositions, $\zeta_p \nin F$, and $\bar \rho_{\pi}^{\oper{ss}}(G_{F(\zeta_p)})$ is adequate.\footnote{The definition of adequate is taken from \cite[Definition 3.1.1]{Allen-SelmerGroups} , which generalizes \cite[Definition 2.3]{Thorne-Automorphy}.} Then $X$ is smooth at $x$.
\end{corollary}

The main corollary is deduced from the theorem as follows. A short computation shows that the contribution of the critical types in Theorem \ref{theorem:intro-theorem} is minimized, and equal to $(F^+:\Q)\cdot n$, exactly when each critical type is a product of distinct simple transpositions. Since $X$ is equidimensional of dimension $(F^+:\Q)\cdot n$, this means that $X$ is regular, and thus smooth, at $x$ in the situation of Corollary \ref{corollary:intro-corollary}. The proof also gives an ``$R = \bf T$'' theorem which we will partially explain in Section \ref{subsec:sketch} below (see Corollary \ref{corollary:final-corollary} for a precise statement). 

Examples constructed in \cite{Bellaiche-Nonsmooth} by Bella\"iche  show that the irreducibility of $\rho_{\pi}$ is important in Theorem \ref{theorem:intro-theorem} and Corollary \ref{corollary:intro-corollary}. Regarding optimality, Breuil, Hellmann and Schraen have shown that Corollary \ref{corollary:intro-corollary} is optimal in that $X$ is singular once one of its critical types is {\em not} a product of distinct simple transpositions. See \cite[Theorem 1.2]{BreuilHellmannSchraen-LocalModel} and compare with the earlier result \cite[Corollary 5.18]{BHS-Classicality}. We expand on these notes following Corollary \ref{corollary:final-corollary} in the text.

\subsection{Sketch of proof}\label{subsec:sketch}
The proof of Theorem \ref{theorem:intro-theorem} follows a well-known strategy: we compare a ring of Hecke operators to a universal deformation ring for a Galois representation. Our ring of Hecke operators is the completion $\hat{\cal O}_{X,x}^{\rig}$ of the rigid analytic local ring of the eigenvariety $X$ at the point $x$, and the deformation ring, denoted $R_{\rho_{\pi}}^{\Ref,\min}$, is a deformation ring for $\rho_{\pi}$. 

The deformations parameterized by $R_{\rho_{\pi}}^{\Ref,\min}$ are {\em weakly-refined} at $p$-adic places and {\em minimally ramified} (or, unramified in the sense of Bloch and Kato) at the places away from $p$. The weakly-refined condition depends on a triangulation (which is suppressed in the notation). When the triangulation is non-critical, the weakly-refined deformations are the same as the {\em trianguline} deformations studied in \cite[Chapter 2]{BellaicheChenevier-Book}.

The interpolation of crystalline periods over eigenvarieties \cite{Kisin-OverconvergentModularForms,Liu-Triangulations} and the minimality of our eigenvariety implies that there is a natural surjective map $R_{\rho_{\pi}}^{\Ref,\min}\surject \hat{\cal O}_{X,x}^{\rig}$. In this way, a bound for the tangent space $T_{X,x}$ is obtained from any bound for the Zariski tangent space $\fr t_{\rho_{\pi}}^{\Ref,\min}$ of the deformation ring $R_{\rho_{\pi}}^{\Ref,\min}$.

Tangent spaces of deformation rings are computed using Galois cohomology. Following an idea of Bella\"iche and Chenevier in the non-critical case, we observe that the global tangent space $\fr t_{\rho_{\pi}}^{\Ref,\min}$ is naturally equipped with restriction maps to the tangent spaces $\fr t_{\rho_{\pi,\twid v}}^{\Ref}$ of the weakly refined deformation problem at the $p$-adic places (see Section \ref{subsec:weakly-refined-deformations}). One has a natural exact sequence
\begin{equation}\label{eqn:intro-sequence}
0 \goto H^1_f(G_{F^+}, \ad \rho_{\pi}) \goto \fr t_{\rho_{\pi}}^{\Ref,\min} \goto \bigdsum_{v \dvd p} \fr t_{\rho_{\pi,\twid v}}^{\Ref}/H^1_{f}(G_{F_{\twid v}},\ad \rho_{\pi,\twid v}).
\end{equation}
where $H^1_f(G_{F_{\twid v}},\ad \rho_{\pi,\twid v})$ is the local Bloch-Kato Selmer group parameterizing infinitesimal crystalline deformations of $\rho_{\pi,\twid v}$.

Our main technical result is a bound on the third term in the sequence \eqref{eqn:intro-sequence}. The following is a purely local theorem, but we state it here in the global context where it is applied.
\begin{theorem}[{Theorem \ref{theorem:best-upper-bound}}]\label{theorem:intro-local-theorem}
For each $v \dvd p$, 
\begin{equation*}
\dim \fr t_{\rho_{\pi,\twid v}}^{\Ref}/H^1_f(G_{F_{\twid v}},\ad \rho_{\pi,\twid v}) \leq \sum_{\tau: F_v^+\goto L} \ell(\sigma_{\twid v,\tau}) + c(\sigma_{\twid v,\tau}).
\end{equation*}
\end{theorem}
Combining the bound in Theorem \ref{theorem:intro-local-theorem} with the sequence \eqref{eqn:intro-sequence} and the preceding paragraphs, we get the bound in Theorem \ref{theorem:intro-theorem}.

In the non-critical case, Theorem  \ref{theorem:intro-local-theorem} is proven in \cite{BellaicheChenevier-Book} by computing the dimension of a trianguline deformation ring. The key point in our generalization is carefully measuring, in terms of the critical type, how far weakly-refined deformations are from being trianguline. For that, we separately study (I) weakly-refined deformations with constant Hodge--Tate weights and (II) the variation of Hodge--Tate weights in weakly-refined deformations.

Versions of the above results were obtained in low-dimensional cases by the author in his Ph.D. thesis \cite{Bergdall-Thesis} and later in unpublished notes. The explicit goal was to prove the smoothness part of Corollary \ref{corollary:intro-corollary}, but only the second half of the computation, referring to (II), was well understood (see \cite[Section 7]{Bergdall-ParabolineVariation} for example).

The condition of the critical types being products of distinct simple transpositions in Corollary \ref{corollary:intro-corollary} did not occur to the author until hearing in lectures at the Centre International de Rencontres Math\'ematiques (Luminy) in 2015 that a similar local result was proven by Breuil, Hellmann and Schraen \cite{BHS-Classicality}. Once the condition was noticed, the statement of Theorem \ref{theorem:intro-local-theorem} and its proof were obtained independently. Compare with \cite[Section 4]{BHS-Classicality}. The applications to eigenvarieties in this paper and in \cite{BHS-Classicality} are different.

\subsection{Organization}
We give a brief reminder on $(\varphi,\Gamma)$-modules in Section \ref{section:reminder}. In Section \ref{section:weakly-refined} we define the weakly-refined deformations and prove Theorem \ref{theorem:intro-local-theorem}. Section \ref{section:application} is dedicated to the proofs of Theorem \ref{theorem:intro-theorem} and Corollary \ref{corollary:intro-corollary}.

\subsection{Notations and conventions}\label{subsec:notations}
We fix an algebraic closure $\bar \Q_p$ and an isomorphism $\bar \Q_p \simeq \C$ which is used implicitly throughout. We assume that $L$ is a finite extension of $\Q_p$ contained in $\bar \Q_p$, and we allow $L$ to change so as to contain the image of any embedding of a $p$-adic field into $\bar \Q_p$.

Suppose $\ell$ is a prime (possibly $\ell = p$) and $K/\Q_\ell$ is a finite extension. We write $K_0$ for the maximal unramified subextension of $K$ and $\ell^{f_K}$ for the number of elements in the residue field of $K$. If $\rho: G_K \goto \GL_n(\bar \Q_p)$ is a continuous representation which is potentially semi-stable (this is automatic if $\ell\neq p$) then we write $\WD(\rho)$ for the corresponding Weil--Deligne representation over $\bar \Q_p$ (see \cite[Theorem 4.2.1]{Tate-NumberTheoreticBackground} if $\ell \neq p$ and \cite{Fontaine-RepresentationSemiStable} if $\ell = p$). As an example, if $\ell = p$ and $\rho$ is crystalline then $\WD(\rho)$ is unramified and the eigenvalues of a geometric Frobenius element are the eigenvalues of the crystalline Frobenius $\varphi^{f_K}$ acting on $D_{\cris}(\rho)$, counted with multiplicity.

Let $\rec_K: K^\times \goto G_K^{\ab}$ be the local Artin reciprocity map, normalized so that the image of a uniformizer corresponds to a geometric Frobenius element. If $\pi$ is an irreducible smooth representation of $\GL_n(K)$ on a $\bar \Q_p$-vector space we denote by $\rec(\pi)$ the $n$-dimensional Frobenius semi-simple Weil--Deligne representation over $\bar \Q_p$ given by the local Langlands correspondence \cite{HarrisTaylor-LocalLanglands}. We normalize the correspondence as follows. Let $T(K) \ci \GL_n(K)$ be the diagonal torus and $B(K) \ci \GL_n(K)$ be the upper triangular  subgroup. If $\chi = \chi_1\otimes \dotsb \otimes \chi_n$ is a character of $T(K)$  then $\rec(\pi(\chi)) = \chi_1\circ \rec_K^{-1}\oplus \dotsb \oplus \chi_n\circ \rec_K^{-1}$, where $\pi(\chi)$ is the unique irreducible unramified subquotient of the smooth, non-normalized, induction $\Ind_{B(K)}^{\GL_n(K)}(\delta_{B(K)}^{1/2}\cdot \chi)$. Here, $\delta_{B(K)} = \abs{-}_K^{n-1}\otimes \dotsb \otimes \abs{-}_K^{1-n}$ is the modulus character of $B(K)$.

If $\ell = p$, and $\delta: \cal O_K^\times \goto L^\times$ is a continuous character then for each embedding $\tau: K \inject L$ we write $\HT_\tau(\delta)$ for the $\tau$-Hodge--Sen--Tate weight of $\delta$, which is the negative of the weight defined in \cite[Definition 6.1.6]{KedlayaPottharstXiao-Finiteness}.

If $n\geq 1$ we let $S_n$ denote the group of permutations on $\set{1,\dotsc,n}$. If $\sigma \in S_n$ we write $\ell(\sigma) = \set{(i,j) \st i < j \text{ and } \sigma(i) > \sigma(j)}$ for its length and $c(\sigma)$ for the number of orbits in $\set{1,\dotsc,n}$ under the action of the group generated by $\sigma$.

\subsection{Acknowledgements}
The author thanks Christophe Breuil, David Hansen, Eugen Hellmann and Benjamin Schraen for helpful discussions and comments. The author also thanks the Centre International de Rencontres Math\'ematiques for hospitality in June 2015. Many ideas in this article were developed when the author was a graduate student at Brandeis University. Thanks are duly given to Jo\"el Bella\"iche for encouragement, numerous insightful conversations and comments on an early draft of this paper. The author was partially supported by NSF award DMS-1402005.

\section{Reminder on $(\varphi,\Gamma)$-modules}\label{section:reminder}

\subsection{$(\varphi,\Gamma_K)$-modules and triangulations}\label{subsec:first-subsec}

Let $K/\Q_p$ be a finite extension. We denote by $\cal R_K$ the Robba ring defined over $K$, i.e. the ring of series $f = \sum a_i T^i$ defined over the maximal absolutely unramified extension of $K_\infty$ and which converge on an annulus $r(f) < \abs{T} < 1$ (see \cite[Section 2]{KedlayaPottharstXiao-Finiteness}). Here $K_\infty$ is the field obtained from by adjoining all the $p$-power roots of unity to $K$. If $A$ is an affinoid $\Q_p$-algebra then we define $\cal R_{K,A} := \cal R_K\hat\tensor_{\Q_p} A$. 

Recall that if $D$ is a finite free module over a commutative ring $R$ then we say a submodule $P \ci D$ is saturated if $D/P$ is projective as an $R$-module. If $L/\Q_p$ is a finite extension then $\cal R_{K,L}$ is an adequate B\'ezout domain \cite[Proposition 4.12]{Berger-Representationp-adique} and so, for $R = \cal R_{K,L}$, projective may be replaced for free. The following lemma is \cite[Lemma 2.2.3]{BellaicheChenevier-Book} when $K = \Q_p$. The proof is no different for general $K$, so we omit it.
\begin{lemma}\label{lemma:saturated-deformation}
Let $A$ be a local Artin $L$-algebra with residue field $L$ and maximal ideal $\ideal m_A$. Suppose that $D$ is a finite free $\cal R_{K,A}$-module which contains a rank one free submodule $P \ci D$. If $P/\ideal m_A P \ci D/\ideal m_A D$ is saturated as an $\cal R_{K,L}$-module then $P$ is saturated in $D$ as well.
\end{lemma}

We equip $\cal R_{K,L}$ with its natural commuting actions of the Frobenius operator $\varphi$ and the group $\Gamma_K = \Gal(K_\infty/K)$ (see \cite[Definition 2.2.2]{KedlayaPottharstXiao-Finiteness}). A $(\varphi,\Gamma_K)$-module $D$ over $\cal R_{K,L}$ is a finite free $\cal R_{K,L}$-module $D$ equipped with commuting $\cal R_{K,L}$-semilinear actions of an operator $\varphi$ and the group $\Gamma_K$, such that $\varphi(D)$ generates $D$ as an $\cal R_{K,L}$-module. For coefficients more general than $L$ (e.g. Artin algebras) see \cite[Chapter 2]{BellaicheChenevier-Book} or \cite[Section 2]{KedlayaPottharstXiao-Finiteness}. The rank of $D$ is the rank of the underlying $\cal R_{K,L}$-module. We write $D^{\dual}$ for the dual $(\varphi,\Gamma_K)$-module.

There is a functor $\rho \mapsto D_{\rig}^{\dagger}(\rho)$ which defines a fully faithful embedding
\begin{equation*}
\set{\text{continuous representations $\rho: G_K \goto \GL_n(L)$}} \hookrightarrow \set{\text{rank $n$ $(\varphi,\Gamma_K)$-modules over $\cal R_{K,L}$}}.
\end{equation*}
Its essential image is the so-called \'etale $(\varphi,\Gamma_K)$-modules characterized using the theory of slope filtrations \cite[Theorem 6.10]{Kedlaya-p-adicMonodromy}. Crucially, $D_{\rig}^{\dagger}(\rho)$ may contain non-\'etale $(\varphi,\Gamma_K)$-submodules even if $\rho$ is irreducible.

Rank one $(\varphi,\Gamma_K)$-modules over $\cal R_{K,L}$ are  classified by continuous characters $\delta: K^\times \goto L^\times$. We write $\cal R_{K,L}(\delta)$ for the $(\varphi,\Gamma_K)$-module corresponding to $\delta$ by \cite[Construction 6.2.4]{KedlayaPottharstXiao-Finiteness}.  If $D$ is a $(\varphi,\Gamma_K)$-module over $\cal R_{K,L}$ then we write $D(\delta) := D\tensor_{\cal R_{K,L}} \cal R_{K,L}(\delta)$ for the ``twist'' of $D$ by $\delta$.

Important constructions in the theory of Galois representations extend to the category of $(\varphi,\Gamma_K)$-modules. For example, a $(\varphi,\Gamma_K)$-module has Galois cohomology $H^{\bullet}(D)$ concentrated in degree at most two (\cite{Herr-Cohomology,Liu-CohomologyDuality}). If $\delta:K^\times \goto L^\times$ is a continuous character we write $H^{\bullet}(\delta)$ in lieu of $H^{\bullet}(\cal R_{K,L}(\delta))$. We also have Fontaine's notions of de Rham, crystalline, etc. for $(\varphi,\Gamma_K)$-modules. For example, $D_{\crys}(D) = D[1/t]^{\Gamma_K}$ where $t \in \cal R_{\Q_p}$ is ``Fontaine's $p$-adic $2\pi i$". (See \cite{Berger-Representationp-adique,Berger-EquationsDifferentielles} for details.)

A triangulation of a $(\varphi,\Gamma_K)$-module $D$ over $\cal R_{K,L}$ is a filtration
\begin{equation*}
P_{\bullet}: 0 = P_0 \sci P_1 \sci \dotsb \sci P_{n-1} \sci P_n = D
\end{equation*}
of $D$ by saturated $(\varphi,\Gamma_K)$-submodules. The parameter of $P_{\bullet}$ is the ordered tuple $(\delta_1,\dotsc,\delta_n)$ of continuous characters $\delta_j$ such that $P_j/P_{j-1} = \cal R_{K,L}(\delta_j)$.

Now let $D$ be a crystalline $(\varphi,\Gamma_K)$-module over $\cal R_{K,L}$. Thus $D_{\cris}(D)$ is a finite free $K_0\tensor_{\Q_p} L$-module equipped with a $K_0$-semilinear (but $L$-linear) operator $\varphi$ and $D_{\cris}(D)_K := D_{\cris}(D)\tensor_{K_0} K$ is equipped with a decreasing, exhaustive and separated filtration $\Fil^{\bullet}$ by $K\tensor_{\Q_p} L$-submodules (the Hodge filtration). The operator $\varphi^{f_K}$ is $K_0$-linear and we refer to its eigenvalues as the crystalline eigenvalues of $D$. Once $L$ is sufficiently large, every crystalline eigenvalue lies in $L^\times \ci (K_0\tensor_{\Q_p} L)^\times$ (compare with the proof of Lemma \ref{lemma:equivalent-conditions}).

If $P_{\bullet}$ is a triangulation of a crystalline $(\varphi,\Gamma_K)$-module $D$ then $D_{\cris}(P_j) \ci D_{\cris}(D)$ is a filtered $\varphi$-submodule of rank $j$. Thus, a triangulation defines an ordering $(\phi_1,\dotsc,\phi_n)$ of the crystalline eigenvalues by declaring the first $j$ eigenvalues appear in $D_{\cris}(P_j)$. If $D$ has distinct crystalline eigenvalues this defines a bijection (see \cite{Berger-EquationsDifferentielles})
\begin{equation}\label{eqn:ref-tri}
\set{\text{triangulations of $D$}} \longleftrightarrow \set{\text{orderings of crystalline eigenvalues for $D$}}.
\end{equation}
Let us briefly recall the ``matching" of weights in this bijection. If $\tau: K \inject L$ is an embedding, the filtration $\Fil^{\bullet}$ on $D_{\cris}(D)_K$ equips $D_{\cris}(D)_K\tensor_{K,\tau} L$ with an exhaustive and separated filtration by $L$-subspaces whose jumps are the $\tau$-Hodge--Tate weights $h_{1,\tau} \leq h_{2,\tau} \leq \dotsb \leq h_{n,\tau}$. Given an ordering $(\phi_1,\dotsc,\phi_n)$ and an embedding $\tau$ we write $(s_{1,\tau},\dotsc,s_{n,\tau})$ for the re-ordering of $\set{h_{i,\tau}}$ such that the induced filtration on $\sum_{i=1}^j D_{\cris}(D)^{\varphi^{f_K}=\phi_i} \tensor_{K_0,\tau} L$ has weights $\set{s_{1,\tau},\dotsc,s_{j,\tau}}$. On the other hand, a triangulation has its parameter $(\delta_1,\dotsc,\delta_n)$ and each character $\delta_i$ has a $\tau$-Hodge--Tate weight $\HT_\tau(\delta_j)$ (see Section \ref{subsec:notations}). The weights match up in that, through \eqref{eqn:ref-tri}, we have $\HT_\tau(\delta_j) = s_{j,\tau}$ for all $\tau$ and $1\leq j \leq n$.

\begin{definition}\label{defi:critical-type}
Let $D$ be a crystalline $(\varphi,\Gamma_K)$-module such that the $\tau$-Hodge--Tate weights $\set{h_{i,\tau}}$ are distinct for each $\tau$ and let $P_{\bullet}$ be a triangulation of $D$ with parameter $(\delta_1,\dotsc,\delta_n)$. The critical type of $P_{\bullet}$ is the collection of permutations $(\sigma_\tau)_{\tau}$ such that $\HT_\tau(\delta_j) = h_{\sigma_\tau(j),\tau}$ for $j=1,\dotsc,n$. We say $P_{\bullet}$ is  non-critical if $\sigma_\tau = \id$ for each $\tau$.
\end{definition}

\begin{remark}
It may be advantageous to see the critical type of a triangulation as lying in the Weyl group of $\Res_{K/\Q_p} \GL_n$ (see \cite{Breuil-LocallyAnalyticSocle2,BHS-Classicality}).
\end{remark}
\begin{example}
Let $(\delta_1,\dotsc,\delta_n)$ be an $n$-tuple of continuous characters $\delta_j : \Q_p^\times \goto L^\times$ such that $\HT(\delta_1) < \dotsb < \HT(\delta_n)$. If $D = \bigdsum_{j=1}^n \cal R_{\Q_p,L}(\delta_j)$ and $\sigma \in S_n$ then the triangulation
\begin{equation*}
(0) \sci \cal R_{K,L}(\delta_{\sigma(1)}) \sci \dotsb \sci \bigdsum_{i=1}^j \cal R_{K,L}(\delta_{\sigma(i)}) \sci \dotsb \sci D
\end{equation*}
has critical type $\sigma$. In particular, only one triangulation of $D$ is non-critical.
\end{example}

\subsection{Deformation theory}
Continue to let $L/\Q_p$ be a finite extension contained in the fixed algebraic closure $\bar \Q_p$. Denote by $\fr{AR}_L$ the category of local Artin $L$-algebras with residue field $L$. For example, the ring of dual numbers $L[\varepsilon] = L[u]/(u^2)$ is in $\fr{AR}_L$. Every element $A \in \fr{AR}_L$ is considered a topological ring with the topology defined by its maximal ideal $\ideal m_A$ and, by definition, morphisms are continuous ring morphisms. A functor $\fr X : \fr{AR}_L \goto \Set$ is (pro)-representable if there exists a complete local noetherian $L$-algebra $R_{\fr X}$ with residue field $L$ and $\Hom_{\cont}(R_{\fr X},A) = \fr X(A)$ for all elements $A \in \fr{AR}_L$.

Let $D$ denote a $(\varphi,\Gamma_K)$-module over $\cal R_{K,L}$. A deformation $D_A$ of $D$ to $A \in \fr{AR}_L$ is a $(\varphi,\Gamma_K)$-module over $\cal R_{K,A}$ together with an isomorphism $\pi : D_A \tensor_A L \simeq D$. An isomorphism between deformations $(D_A,\pi)$ and $(D_A',\pi')$ is a $(\varphi,\Gamma_K)$-equivariant isomorphism $\alpha: D_A \overto{\simeq} D_A'$ such that $\pi' \compose \alpha = \pi$. Thus, we have a functor
\begin{equation*}
\fr X_D(A) :=\set{ \text{isomorphism classes of deformations of $D$ to $A$}},
\end{equation*}
which we call the universal deformation functor of $D$. We have that $\fr X_D(L) = \set{D}$ is a single point, and the universal deformation functor of $D$ admits a well-defined, finite-dimensional, Zariski tangent space $\fr t_D := \fr X_D(L[\varepsilon])$ (see \cite[Sections 18 and 23]{Mazur-Fermat-Deformations} and \cite[Proposition 3.4]{Chenevier-InfiniteFern}).

The tangent space $\fr t_D$ admits a canonical description in terms of Galois cohomology. Namely, a deformation $\twid D$ of $D$ to $L[\varepsilon]$ is naturally an extension $0 \goto D \goto \twid D \goto D \goto 0$ in the category of $(\varphi,\Gamma_K)$-modules over $\cal R_{K,L}$, where the submodule is $\varepsilon \twid D$ and the quotient is $\twid D/\varepsilon \twid D$. The association of $\twid D$ to its extension class defines a canonical $L$-linear isomorphism
\begin{equation*}
\fr t_D \simeq \Ext^1_{(\varphi,\Gamma_K)}(D,D) \simeq  H^1(\ad D).
\end{equation*}
See \cite[Proposition 3.6(ii)]{Chenevier-InfiniteFern} for a direct construction going from $\fr t_D$ to $H^1(\ad D)$.

If $\fr X' \ci \fr X$ is an inclusion of functors on $\fr{AR}_L$ then we recall that there is a notion of $\fr X'$ being relatively representable over $\fr X$ (see \cite[Section 19]{Mazur-Fermat-Deformations}). Subfunctors $\fr X' \ci \fr X_D$ are relatively representable if and only if for every morphism of functors $\fr X_D \goto \fr Y$ with $\fr Y$ representable, the base change $\fr Y' := \fr X' \times_{\fr X_D} \fr Y$ is representable. In particular being relatively representable is stable under products over $\fr X_{D}$.


\subsection{Deformations of algebraic characters}\label{subsec:algebraic-char-deformations}
If $A$ is an affinoid $\Q_p$-algebra then let $\cal W(A) = \Hom_{\cont}((\cal O_K^\times)^{n}, A^\times)$. This defines a rigid analytic space over $\Q_p$ (a disjoint union of polydiscs) called the $p$-adic weight space of ${\GL_n}_{/K}$. Since $\cal W$ is smooth over $\Q_p$, if $A \in \fr{AR}_L$ then the canonical morphism $\cal W(A) \goto \cal W(L)$ is surjective. Here we explicitly describe the preimage of $\Q_p$-algebraic elements of $\cal W(L)$.

Let $\underline h = (h_\tau)_{\tau}$ be a collection of integers and $z^{\underline h} : K^\times \goto L^\times$ be given by
\begin{equation*}
z \overset{z^{\underline h}}{\longmapsto} \prod_{\tau} \tau(z)^{h_\tau}.
\end{equation*}
The character $z^{\underline h}$ has Hodge--Tate weights $(-h_\tau)_\tau$. Every $\Q_p$-algebraic character of $K^\times$ is $z^{\underline h}$ for some $\underline h$. We abuse notation and write $\underline h$ for $z^{\underline h}$ as well.\footnote{Warning: if $D$ is a $(\varphi,\Gamma_K)$-module then $D(\underline h)$ is {\em not} necessarily a Tate twist.}

We now restrict to $\cal O_K^\times$. We have that $\cal O_K^\times = \mu \times U$ where $\mu$ is the torsion subgroup and $U \ci \cal O_K^\times$ is a finite free $\Z_p$-submodule. If $\alpha \in \cal O_K^\times$ we write $\alpha = \omega(\alpha)\langle \alpha \rangle$ where $\langle \alpha \rangle \in U$ and $\omega(\alpha) \in \mu$. Let $\underline h$ be as above and suppose $A \in \fr{AR}_L$. Suppose that for each $\tau$ we choose $\eta_\tau \in A$ such that $\eta_\tau \congruent h_\tau \bmod \ideal m_A$. Then, we define a character $\underline\eta: \cal O_K^\times \goto A^\times$ by
\begin{equation*}
z \overset{\underline \eta}{\longmapsto} \prod_{\tau} \tau(z)^{\eta} := \prod_{\tau} \exp\bigl((\eta_\tau -h_\tau)\cdot \tau(\log\langle z\rangle)\bigr) \cdot \tau(z)^{h_\tau}
\end{equation*}
(the exponential converges in $A^\times$ because $\eta_\tau - h_\tau \in \ideal m_A$).  The preimage of $\underline h$ under $\cal W(A) \goto \cal W(L)$ is exactly the set of characters $\underline \eta$.

Note that $\underline h$ is a character of $K^\times$, but extending $\underline \eta$ to $K^\times$ requires a choice. After fixing a uniformizer $\varpi_K \in K^\times$, we denote by $\underline \eta_{\varpi_K}$ the character of $K^\times$ which acts as $\underline \eta$ on $\cal O_K^\times$ and sends $\varpi_K$ to $1$. This defines a rank one $(\varphi,\Gamma_K)$-module $\cal R_{K,A}(\underline \eta_{\varpi_K})$ over $\cal R_{K,A}$ whose Hodge--Sen--Tate  weights are $(-\eta_\tau)_{\tau}$. The character $\underline \eta_{\varpi_K}$ depends on $\varpi_K$, but if $\varpi_K'$ is another choice of uniformizer then $\underline \eta_{\varpi_K}\cdot \underline \eta_{\varpi_K'}^{-1}$ is crystalline.

\subsection{Hodge--Tate deformations}\label{subsec:hodge-tate-deformations}
Let $D$ be a $(\varphi,\Gamma_K)$-module over $\cal R_{K,L}$. We assume $D$ is Hodge--Tate with distinct Hodge--Tate weights. Let $A \in \fr{AR}_L$ and $D_A$ be a deformation of $D$ to $A$. For each embedding $\tau: K \inject L$, the $\tau$-Sen polynomial $P^{\sen,\tau}_{D_A}(x) \in A[x]$ has distinct roots modulo $\ideal m_A$ and so $P^{\sen,\tau}_{D_A}$ completely factors over $A$ by Hensel's lemma. Thus for each $\tau$-Hodge--Tate weight $h_\tau \in \Z$ of $D$ there is a unique $\tau$-Hodge--Sen--Tate weight $\eta_\tau \in A$ of $D_A$ such that $\eta_\tau \congruent h_\tau \bmod \ideal m_A$. We say $h_\tau$ is a constant weight if $\eta_\tau = h_\tau$. Then, we define
\begin{equation*}
\fr X_D^{h_\tau}(A) = \set{\text{deformations $D_A$ of $D$ to $A$ such that the weight $h_\tau$ is constant}}.
\end{equation*}
By definition this defines a subfunctor $\fr X_D^{h_\tau} \ci \fr X_D$. If $\set{h_{i,\tau}}$ is the set of all the Hodge--Tate weights of $D$ then the Hodge--Tate deformation functor is defined to be
\begin{equation*}
\fr X_D^{\HT} :=  \bigintersect_{i,\tau} \fr X_D^{h_{i,\tau}}.
\end{equation*}
A deformation $D_A$ is Hodge--Tate if and only if its Hodge--Sen--Tate weights are constant (integers), if and only if it is Hodge--Tate as a $(\varphi,\Gamma_K)$-module over $\cal R_{K,A}$.
\begin{proposition}\label{prop:const-wt-rep}
The subfunctor $\fr X_D^{h_{i,\tau}} \ci \fr X_D$ is relatively representable for each $i,\tau$. Thus, $\fr X_D^{\HT} \ci \fr X_D$ is relatively representable as well.
\end{proposition}
\begin{proof}
Being relatively representable is closed under intersection, so only the first statement needs proving. Observe that $D_A \in \fr X_D^{h_{i,\tau}}(A)$ if and only if the Sen operator acts semi-simply on the generalized eigenspace for the eigenvalue $h_{i,\tau}$ inside $D_{\sen}(D_A)$ (viewing $D_A$ over $\cal R_{K,L}$) . Since this property is closed under subquotients and direct sums, the relative representability of $\fr X_{D}^{h_{i,\tau}}$ follows from Ramakrishna's criterion \cite[Section 25]{Mazur-Fermat-Deformations}.
\end{proof}


\subsection{Crystalline deformations}\label{subsec:crystalline-deformations}
Throughout this section we denote by $D$ a crystalline $(\varphi,\Gamma_K)$-module over $\cal R_{K,L}$. The fine Selmer group for $D$, generalizing the corresponding notion for Galois representations \cite{BlochKato-TamagawaNumbersOfMotives}, is defined by
\begin{equation*}
H^1_f(D) := \ker\left(H^1(D) \goto H^1(\Gamma_K, D[t^{-1}])\right).
\end{equation*}
The cohomology on the right is the continuous cohomology of the profinite group $\Gamma_K$. We refer to \cite[Section 3.1]{Pottharst-FiniteSlope} and \cite[Section 1]{Benois-GreenbergL} for the facts that follow.

First, by \cite[Equation 3-2]{Pottharst-FiniteSlope} the dimension of $H^1_f(D)$ is computed by
\begin{equation}\label{eqn:selmer-dimension}
\dim_L H^1_f(D) = \dim_L H^0(D) + \dim_L D_{\dR}(D)/D_{\dR}^+(D).
\end{equation}
Second, fine Selmer groups arise naturally as tangent spaces to a deformation problem. Indeed, we define the crystalline deformation functor as
\begin{equation*}
\fr X_{D,f}(A) = \set{D_A \in \fr X_D(A) \st \text{$D_A$ is crystalline}}.
\end{equation*}
Since being crystalline is closed under direct sums and subquotients, Ramakrishna's criterion \cite[Section 25]{Mazur-Fermat-Deformations} implies that $\fr X_{D,f} \ci \fr X_D$ is relatively representable.  We write  $\fr t_{D,f} := \fr X_{D,f}(L[\varepsilon])$ for the Zariski tangent space to $\fr X_{D,f}$.

Since $D$ is crystalline, so is $\ad D$ and the Galois cohomology $H^1(\ad D)$ contains the subspace $H^1_f(\ad D)$. It is explained in \cite[Section 3]{Pottharst-FiniteSlope} that an extension class
\begin{equation*}
0 \goto D \goto \twid D \goto D \goto 0
\end{equation*}
in $\Ext^1_{(\varphi,\Gamma_K)}(D,D) \simeq H^1(\ad D)$ lies in $H^1_f(\ad D)$ if and only if $\twid D$ is crystalline. If $\Ext^1_f(D,D)$ denotes the corresponding subspace of crystalline extensions then the inclusion $\fr t_{D,f} \ci \fr t_D$ induces isomorphisms $\fr t_{D,f} \simeq \Ext^1_f(D,D) \simeq H^1_f(\ad D)$.

We record here a property of Selmer groups. If $D$ is a crystalline $(\varphi,\Gamma_K)$-module then we write $H^1_{/f}(D) := H^1(D)/H^1_f(D)$.
\begin{lemma}\label{lemma:exactness-selmer-groups}
If $0 \goto D_1 \goto D_2 \goto D_3 \goto 0$ is a short exact sequence of crystalline $(\varphi,\Gamma_K)$-modules and $H^2(D_1) = 0$ then the canonical morphisms induce a short exact sequence
\begin{equation*}
0 \goto H^1_{/f}(D_1) \goto H^1_{/f}(D_2) \goto H^1_{/f}(D_3) \goto 0.
\end{equation*}
\end{lemma}
\begin{proof}
Consider the commuting diagram
\begin{equation*}\label{eqn:benois}
\xymatrix{
H^0(D_3) \ar@{=}[d] \ar[r] & H^1_f(D_1) \ar@{^{(}->}[d] \ar[r] & H^1_f(D_2)  \ar@{^{(}->}[d] \ar[r] & H^1_f(D_3) \ar@{^{(}->}[d] \ar[r] & 0\\
H^0(D_3) \ar[r] & H^1(D_1) \ar[r] & H^1(D_2) \ar[r] & H^1(D_3) \ar[r] & 0.
}
\end{equation*}
The top row is exact by \cite[Corollary 1.4.6]{Benois-GreenbergL}. The bottom row is exact because $H^2(D_1) = (0)$. From the snake lemma we get a short exact sequence
\begin{equation}\label{eqn:first-term}
0 \goto {\coker(H^0(D_3) \goto H^1(D_1))\over \coker(H^0(D_3) \goto H^1_f(D_1))} \goto H^1_{/f}(D_2) \goto H^1_{/f}(D_3) \goto 0.
\end{equation}
Since the first term of \eqref{eqn:first-term} equals $H^1_{/f}(D_1)$, we are finished.
\end{proof}

\section{Weakly-refined deformations}\label{section:weakly-refined}

\subsection{Deforming crystalline eigenvalues}\label{subsec:deform-crys}

Let $D$ be a crystalline $(\varphi,\Gamma_K)$-module over $\cal R_{K,L}$. We assume that $h_{1,\tau} = 0$ is the unique least $\tau$-Hodge--Tate weight for each embedding $\tau: K \goto L$. Following Section \ref{subsec:hodge-tate-deformations} we let $\fr X_D^0 = \bigintersect_\tau \fr X_{D}^{0_\tau}$ be the relatively representable subfunctor of deformations with constant Hodge--Tate weight zero at each embedding. 

Suppose that $\Phi$ is a crystalline eigenvalue for $D$ (note that it appears in $D_{\crys}(D) = D_{\crys}^+(D) = D^{\Gamma_K}$). Now set
\begin{equation}\label{eqn:kisin-functor-def}
\fr X_D^{0,\Phi}(A) = \set{\text{\parbox{11cm}{\centering $D_A \in \fr X_D^0(A)$ such that $D_{\crys}^+(D_A)^{\varphi^{f_K} = \Phi_A}$ is free of rank one over $K_0 \tensor_{\Q_p} A$ for some $\Phi_A \in A^\times$ and $\Phi_A \congruent \Phi \bmod \ideal m_A$}}}.
\end{equation}
This clearly defines a subfunctor of $\fr X_D^{0,\Phi} \ci \fr X_D$. 
\begin{proposition}\label{prop:kisin-rel-rep}
If $\Phi$ is a simple crystalline eigenvalue for $D$ then $\fr X_D^{0,\Phi} \ci \fr X_D$ is relatively representable.
\end{proposition}
Kisin \cite[Proposition 8.13]{Kisin-OverconvergentModularForms} and Tan \cite[Section 5.2]{Tan-Thesis} proved Proposition \ref{prop:kisin-rel-rep} for Galois representations. The rest of this subsection is devoted to a proof in the setting of $(\varphi,\Gamma_K)$-modules, adapting \cite[Section 2.3]{BellaicheChenevier-Book}. After twisting $D$ by an unramified character, we may assume that $\Phi = 1$. To emphasize this we write $\fr X_{D}^{0,1} = \fr X_{D}^{0,\Phi=1}$.

If $E$ is a $(\varphi,\Gamma_K)$-module over $\cal R_{K,L}$, we set
\begin{equation*}
F(E) = \set{e \in E^{\Gamma_K} \st (\varphi^{f_K}-1)^ne = (0) \text{ for some $n\geq 1$}} \ci D_{\crys}^+(E).
\end{equation*}
The functor $F(-)$ is left exact, and since $\Phi=1$ is a simple eigenvalue of $D$ we have $\dim_L F(D) = (K_0:\Q_p)$.

Let $A$ denote an element in $\fr{AR}_L$. If $D_A$ is a $(\varphi,\Gamma_K)$-module over $\cal R_{K,A}$ and $M$ is an $A$-module then we consider $D_A \tensor_A M$ as a $(\varphi,\Gamma_K)$-module over $\cal R_{K,L}$ with the trivial actions on $M$.  The $L$-vector space $F(D_A\tensor_A M)$ is naturally equipped with the structure of an $A$-module. We let $\ell_A$ denote the length function on $A$-modules. 

\begin{lemma}\label{lemma:length-lemma}
Let $A \in \fr{AR}_L$ and $D_A \in \fr X_D(A)$. If $M$ is a finite length $A$-module then $\ell_A(F(D_A\tensor_A M))\leq (K_0:\Q_p)\ell_A(M)$. In particular, if $I \ci A$ is an ideal then $\ell_A(F(ID_A))\leq (K_0:\Q_p) \ell_A(I)$.
\end{lemma}
\begin{proof}
The  second statement follows from the first since $D_A$ is flat over $A$ (it is even free), so $ID_A\simeq D_A\tensor_A I$. The first statement is an immediate d\'evissage using the left-exactness of $F(-)$ and the fact explained above that $\ell_A(F(D)) = \dim_L F(D) = (K_0:\Q_p)$.
\end{proof}

\begin{lemma}\label{lemma:equivalent-conditions}
Let $D_A \in \fr X_D(A)$. The following are equivalent
\begin{enumerate}
\item $D_A \in \fr X_{D}^{0,\Phi=1}(A)$.
\item $F(D_A)$ is free of rank one over $K_0\tensor_{\Q_p} A$.
\item $\ell_A(F(D_A)) = (K_0:\Q_p)\ell_A(A)$.
\item The natural map $F(D_A)/\ideal m_AF(D_A) \goto F(D)$ is an isomorphism.
\end{enumerate}
\end{lemma}
\begin{proof}
First assume that $D_A \in \fr X_{D}^{0,\Phi=1}(A)$ and let $\Phi_A \in A$ be the deformation of $\Phi=1$ as in the definition \eqref{eqn:kisin-functor-def}. Then $D_{\crys}^+(D_A)^{\varphi^{f_K} = \Phi_A}$ is a free, rank one, $K_0\tensor_{\Q_p} A$-submodule of $F(D_A)$. It cannot be proper since that would imply $\ell_A(F(D_A)) > \ell_A(D_{\crys}^+(D_A)^{\varphi^{f_K} = \Phi_A}) =\ell_A(K_0\tensor_{\Q_p} A)$, which would contradict Lemma \ref{lemma:length-lemma} (applied with $M = A$). Thus, $F(D_A) = D_{\crys}^+(D_A)^{\varphi^{f_K} = \Phi_A}$ is free of rank one over $K_0\tensor_{\Q_p} A$. This proves (a) implies (b).

We clearly have (b) implies (c). Now we show (c) implies (d). Since $F(-)$ is left exact we have an exact sequence 
\begin{equation}\label{eqn:length-consider}
0 \goto F(\ideal m_AD_A) \goto F(D_A) \goto F(D).
\end{equation}
Under the assumption (c), considering the lengths in \eqref{eqn:length-consider},  Lemma \ref{lemma:length-lemma} implies that $F(D_A) \goto F(D)$ is surjective. But then it follows that  $F(D_A)/\ideal m_A F(D_A) \goto F(D)$ is onto as well. Since the two $A$-modules have the same length, we've proven (d).

It remains to prove (d) implies (a) (we will simultaneously show (d) implies (b)). First, we can choose a vector $v \in F(D_A)$ such that the image in $F(D)$ is a $K_0\tensor_{\Q_p} L$-module basis. It follows that the $K_0\tensor_{\Q_p} A$-module spanned by $v$ inside $F(D_A)$ is free of rank one.\footnote{This is a minor modification of an argument given by Bella\"iche and Chenevier in the proof of \cite[Proposition 2.3.9]{BellaicheChenevier-Book}. To reduce to their argument, use that the image of $v$ is non-zero in $F(D) \tensor_{K_0,\tau} L$ for each embedding $\tau: K_0 \inject L$.} By considering lengths, Lemma \ref{lemma:length-lemma} implies that the containment $(K_0\tensor_{\Q_p} A)\cdot v \ci F(D_A)$ is necessarily an equality. This shows (d) implies (b) and since $F(D_A)$ is $\varphi$-stable, $\varphi(v) = xv$ for some $x \in K_0\tensor_{\Q_p} A$. But then, $\varphi^{f_K}(v) = x\varphi(x)\dotsb \varphi^{f_K-1}(x)v$ where the $\varphi$ on the right-hand side is the natural $A$-linear Frobenius on $K_0\tensor_{\Q_p} A$. Since $x\varphi(x)\dotsb \varphi^{f_K-1}(x)$ lies in $A^\times \ci K_0\tensor_{\Q_p} A$ we will call it $\Phi_A$. Note that since $v$ spans $F(D)$ modulo $\ideal m_A$, we know that $\Phi_A \congruent 1 \bmod \ideal m_A$. Finally, since $v \in F(D_A) \ci D_A^{\Gamma_K} = D_{\crys}^+(D_A)$ as well, we've shown that $v \in D_{\crys}(D_A)^{\varphi^{f_K} = \Phi_A}$. But then we have $(K_0\tensor_{\Q_p} A)\cdot v = D_{\crys}^+(D_A)^{\varphi^{f_K} = \Phi_A} = F(D_A)$ and we've shown (d) implies (a).
\end{proof}
We are now ready to prove Proposition \ref{prop:kisin-rel-rep}.
If $D_A \in \fr X_D(A)$ and $A \goto A'$ is an arrow in $\fr{AR}_L$ then we write $D_{A'} := D_A \tensor_A A'$.
\begin{proof}[Proof of Proposition \ref{prop:kisin-rel-rep} with $\Phi=1$]
To show $\fr X_D^{0,\Phi=1} \ci \fr X_D$ is relatively representable we will use Schlessinger's criterion \cite[Section 23]{Mazur-Fermat-Deformations}, labeled by conditions (1), (2), and (3) in {\em loc. cit.} Condition (1) is clear.  Given (1) and (3) in {\em loc. cit.} it is enough to check condition (2) with $C = L$, but that it is also clear (compare with \cite[Propositions 8.13 and 8.7]{Kisin-OverconvergentModularForms}).

It remains to check condition (3): {\em if $A \ci A'$ is an inclusion in $\fr{AR}_L$, $D_A \in \fr X_D(A)$ and $D_{A'} \in \fr X_{D}^{0,\Phi=1}(A')$ then $D_A \in \fr X_D^{0,\Phi=1}(A)$}. We start with the assumptions of condition (3) and consider the exact sequence $0 \goto F(D_A) \goto F(D_{A'}) \goto F(D_A \tensor_A A'/A)$. Computing lengths,
\begin{align}\label{eqn:first-inequality-blah}
\ell_A(F(D_{A'})) - (K_0:\Q_p)\ell_A(A'/A) &\leq \ell_A(F(D_{A'})) - \ell_A(F(D_A\tensor_A A'/A))\\
& \leq \ell_A(F(D_A))\nonumber
\end{align}
Here we used Lemma \ref{lemma:length-lemma} for the first inequality. Since $D_{A'} \in \fr X_{D}^{0,\Phi=1}(A')$ we have that $F(D_{A'})$ is free of rank one over $K_0\tensor_{\Q_p} A'$ and thus \eqref{eqn:first-inequality-blah} implies $(K_0:\Q_p)\ell_A(A) \leq \ell_A(F(D_A))$. The reverse inequality is also try by Lemma \ref{lemma:length-lemma}. So, we have shown that $\ell_A(F(D_A)) = (K_0:\Q_p)\ell_A(A) = \ell_A(K_0\tensor_{\Q_p} A)$. But then $D_A \in \fr X_D^{0,\Phi=1}(A)$ by Lemma \ref{lemma:equivalent-conditions}.
\end{proof}

\subsection{Weakly-refined deformations}\label{subsec:weakly-refined-deformations}

We temporarily fix a uniformizer $\varpi_K \in K^\times$. We also assume $D$ is a crystalline $(\varphi,\Gamma_K)$-module over $\cal R_{K,L}$ and that for each embedding $\tau: K \inject L$ there is a unique least $\tau$-Hodge--Tate weight $h_{1,\tau}$. 

If $A \in \fr{AR}_L$ and $D_A$ is a deformation of $D$ to $A$ we write $\eta_{1,\tau}$  for the Hodge--Sen--Tate weight of $D_A$ satisfying $\eta_{1,\tau} \congruent h_{1,\tau} \bmod \ideal m_A$ (cf.\ Section \ref{subsec:hodge-tate-deformations}). Then, by Section \ref{subsec:algebraic-char-deformations} we can construct a character $\underline {\eta_1}_{\varpi_K}: K^\times \goto A^\times$ whose value on $\varpi_K$ is 1 and whose $\tau$-Hodge--Sen--Tate weight is $\eta_{1,\tau}$ for each embedding $\tau: K \inject L$. The twisting operation $D_A \mapsto D_A(\underline{\eta_1}_{\varpi_K})$ defines a natural transformation $\fr X_D \goto \fr X_{D(\underline{h_1}_{\varpi_K})}^0$. (Fixing $\varpi_K$ makes the twisting functorial.)

Continue with the assumptions of the previous paragraphs. If $\phi$ is a simple crystalline eigenvalue for $D$ then $\Phi := \phi\prod_{\tau} \tau(\varpi_K)^{-h_{1,\tau}}$ is a simple eigenvalue for $D(\underline{h_1}_{\varpi_K})$. Note that the least Hodge--Tate weight of $D(\underline{h_1}_{\varpi_K})$ is zero at each embedding $\tau$. We now define a functor $\fr X_D^{\phi}$ as a fibered product
\begin{equation*}
\xymatrix{
\fr X_D^{\phi} \ar[d] \ar[r] & \fr X_D \ar[d]^{D_A\mapsto D_A(\underline{\eta_1}_{\varpi_K})}\\
\fr X_{D(\underline{h_1}_{\varpi_K})}^{0,\Phi} \ar[r] & \fr X_{D(\underline{h_1}_{\varpi_K})}^0
}
\end{equation*}
where the bottom arrow is the natural inclusion.

Since being relatively representable is stable under base change, Proposition \ref{prop:kisin-rel-rep} implies that $\fr X_D^{\phi} \ci \fr X_D$ is relatively representable.  Furthermore, $\fr X_D^{\phi}$ is independent of $\varpi_K$: the twisting $D_A \mapsto D_A(\underline{\eta_1}_{\varpi_K})$ and $\Phi$ depend on the choice but the dependencies cancel.\footnote{As we mentioned at the end of Section \ref{subsec:algebraic-char-deformations}, if $\varpi_K$ and $\varpi_K'$ are two different uniformizers then $\underline{\eta_1}_{\varpi_K}$ and $\underline{\eta_1}_{\varpi_K'}$ differ by a crystalline character. Moreover, the crystalline eigenvalue exactly matches the difference between $\Phi$ and the corresponding $\Phi'$.}

For the rest of this subsection, we assume that $D$ is a crystalline $(\varphi,\Gamma_K)$-module over $\cal R_{K,L}$ with distinct Hodge--Tate weights and distinct crystalline eigenvalues. We also equip $D$ with a triangulation $P_{\bullet}$ whose parameter we write $(\delta_1,\dotsc,\delta_n)$. We denote by $(\phi_1,\dotsc,\phi_n)$ the corresponding list of crystalline eigenvalues given by \eqref{eqn:ref-tri}.
%

If $1 \leq j \leq n$ then consider the $j$th exterior power $\wedge^j D$ of $D$. If $D_A$ is a deformation of $D$ to $A$ then $\wedge^j D_A$ is a $(\varphi,\Gamma_K)$-module over $\cal R_{K,A}$ which deforms $\wedge^j D$. If the $\tau$-Hodge--Tate weights of $D$ are $h_{1,\tau} < h_{2,\tau} < \dotsb < h_{n,\tau}$ then the $\tau$-Hodge--Tate weight $h_{1,\tau} + \dotsb + h_{j,\tau}$ for $\wedge^j D$ is the unique least $\tau$-Hodge--Tate weight for each $\tau$. Moreover, if $1 \leq j \leq n$ then $\phi_1\dotsb \phi_j$ is an eigenvalue for $\varphi^{f_K}$ acting on $D_{\crys}(\wedge^j D)$.

\begin{definition}\label{defi:regular-generic}
A triangulation $P_{\bullet}$ is called regular if $\phi_1\dotsb \phi_j$ is a simple crystalline eigenvalue in $D_{\cris}(\wedge^j D)$ for $1 \leq j \leq n$. We say $P_{\bullet}$ is regular generic if $P_{\bullet}$ is regular and in addition $H^2(\delta_i\delta_j^{-1}) = (0)$ for all $1\leq i, j\leq n$.
\end{definition}

\begin{remark}
If $P_{\bullet}$ is a regular triangulation as in Definition \ref{defi:regular-generic} then the crystalline eigenvalues of $D$ are necessarily distinct. Since $H^0(\delta_i\delta_j^{-1}) \subset D_{\crys}(\mathcal R_{K,L}(\delta_i\delta_j^{-1}))^{\varphi=1}$, it follows that $H^0(\delta_i\delta_j^{-1}) = (0)$ when $i \neq j$.
\end{remark}

\begin{definition}\label{def:weakly-refined}
Let $D$ be a crystalline $(\varphi,\Gamma_K)$-module with distinct Hodge--Tate weights. If $P_{\bullet}$ is a regular triangulation then the weakly-refined deformation functor with respect to $P_{\bullet}$ is the fibered product $\fr X_{D,P_{\bullet}}^{\Ref}$ given by
\begin{equation*}
\xymatrix{
\fr X_{D,P_{\bullet}}^{\Ref} \ar[r] \ar[d] & \ar[d]^-{D_A \mapsto (\wedge^j D_A)_j} \fr X_D\\
\prod_{j=1}^n \fr X_{\wedge^j D}^{\phi_1\dotsb \phi_j} \ar[r] & \prod_{j=1}^n \fr X_{\wedge^j D} 
}
\end{equation*}
where the bottom horizontal arrow is the natural inclusion.
\end{definition}

When $P_{\bullet}$ is a regular triangulation, Proposition \ref{prop:kisin-rel-rep} implies that each arrow $\fr X_{\wedge^j D}^{\phi_1\dotsb \phi_j} \goto \fr X_{\wedge^j D}$ is relatively representable. Thus, so is $\fr X_{D,P_{\bullet}}^{\Ref} \ci \fr X_D$.

\subsection{A constant weight tangent space}

Our goal for the rest of Section \ref{section:weakly-refined} is to compute the Zariski tangent spaces $\fr t_{D,P_{\bullet}}^{\Ref}$ to functors of the form $\fr X_{D,P_{\bullet}}^{\Ref}$. This particular subsection concerns the constant weight subfunctor $\fr X_{D,P_{\bullet}}^{\Ref, \HT} := \fr X_{D,P_{\bullet}}^{\Ref} \intersect \fr X_D^{\HT}$, which admits a simpler description: if $D_A$ is a Hodge--Tate deformation of $D$ then the twisting character(s) denoted $\underline{\eta_1}_{\varpi_K}$ are crystalline, and thus the points of $\fr X_D^{\Ref,\HT}$ are given by
\begin{equation}\label{eqn:HT-defs}
\fr X_D^{\Ref,\HT}(A) = \set{\text{\parbox{12cm}{\centering $D_A \in \fr X_D^{\HT}(A)$ such that for some collection $\phi_{j,A} \congruent \phi_j \bmod \ideal m_A$, $D_{\crys}(\wedge^j D_A)^{\varphi^{f_K} = {\phi_{1,A}\dotsb \phi_{j,A}}}$ is free of rank one over $K_0\tensor_{\Q_p} A$ for each $1 \leq j \leq n$.} }}.
\end{equation}

From the description \eqref{eqn:HT-defs} it is clear that $\fr X_{D,P_{\bullet}}^{\Ref,\HT}$ contains the crystalline deformation functor $\fr X_{D,f}$ for any $P_{\bullet}$. In Theorem \ref{theorem:HT-ref-bound} below we compute a bound for the dimension of quotient spaces of the form $\fr t_{D,P_{\bullet}}^{\Ref,\HT}/\fr t_{D,f}$. Let us preview our theorem by remarking on the non-critical case. 

\begin{remark}
Suppose that $D$ is a crystalline $(\varphi,\Gamma_K)$-module equipped with a regular generic triangulation $P_{\bullet}$ which is {\em non-critical}. Then, every weakly-refined deformation is trianguline \cite[Theorem 2.5.6]{BellaicheChenevier-Book} and every trianguline deformation of constant Hodge--Tate weight is crystalline \cite[Theorem 2.5.1]{BellaicheChenevier-Book}. Thus $\fr t_{D,P_{\bullet}}^{\Ref,\HT}/\fr t_{D,f} = (0)$ in this case.
\end{remark}
The next three lemmas expand on the previous remark for triangulations that are possibly critical, and our proof of Theorem \ref{theorem:HT-ref-bound} gives a new proof of the remark.

\begin{lemma}\label{lemma:diagram-chase}
Suppose that $D$ is a $(\varphi,\Gamma_K)$-module over $\cal R_{K,L}$ and $\iota: P_1 \inject D$ is a saturated $(\varphi,\Gamma_K)$-submodule. Let $Q_1 := \coker(\iota)$. Consider the canonical morphisms
\begin{equation*}
\xymatrix{
\Ext^1_{(\varphi,\Gamma_K)}(Q_1,D) \ar[d]_-{h} \ar[r]^-{f} & \Ext^1_{(\varphi,\Gamma_K)}(D,D)\\
\Ext^1_{(\varphi,\Gamma_K)}(Q_1,Q_1).
}
\end{equation*}
\begin{enumerate}
\item If $\twid D = f(X)$ for some $X \in \Ext^1_{(\varphi,\Gamma_K)}(Q_1,D)$ then there exists a saturated embedding $P_1\tensor_{\cal R_{K,L}} \cal R_{K,L[\varepsilon]} \overset{\twid\iota}{\inject} \twid D$ such that $\twid \iota \congruent \iota \bmod \varepsilon$.
\item If $\dim_L \Hom_{(\varphi,\Gamma_K)}(P_1,D) = 1$ and $\twid D = f(X)$ for some $X \in \Ext^1_{(\varphi,\Gamma_K)}(Q_1,D)$ then $\twid \iota$ in part (a) is unique up to a scalar $1 + \mu\varepsilon \in L[\varepsilon]^\times$ and  $h(X) = \coker(\twid \iota)$.
\end{enumerate}
\end{lemma}
\begin{proof}
We first recall the definitions of $f$ and $h$. Let
\begin{equation*}
0 \goto D \overto{j_X} X \overto{\pi_X} Q_1 \goto 0
\end{equation*}
be an extension in $\Ext^1_{(\varphi,\Gamma_K)}(Q_1,D)$. Then, the embedding $\iota: P_1\inject D$ induces an embedding $j_X\compose \iota: P_1 \inject X$ and $h(X) := \coker(j_X\compose \iota)$. If $\theta: D \goto Q_1$ is the quotient map then $f(X)$ is defined by
\begin{equation*}
f(X) := \ker\left(X\oplus D \overset{(\pi_X,-\theta)}{\longrightarrow} Q_1\right) = \set{(x,d) \in X\oplus D \st \pi_X(x) = \theta(d)}.
\end{equation*}
The $(\varphi,\Gamma_K)$-module structure on $f(X)$ is coordinate-wise, and the structure of multiplication by $\varepsilon$ on $f(X)$ is given by $\varepsilon\cdot (x,d) = (j_X(d), 0)$.

Now let's prove (a). Let $X \in \Ext^1_{(\varphi,\Gamma_K)}(Q_1,D)$ and $\twid D = f(X)$. We have a canonical $(\varphi,\Gamma_K)$-equivariant  embedding $P_1 \inject D \inject X\oplus D$ given by $\twid \iota_0(p) = (0,\iota(p))$. Since $\pi_X(0) = 0 = \theta\iota(p)$ we see that $\twid\iota_0(P_1) \ci \twid D$. We define $\twid \iota$ on $P_1\tensor_{\cal R_{K,L}} \cal R_{K,L[\varepsilon]}$ by $\twid \iota(p) = \twid \iota_0(p)$ and $\twid \iota(\varepsilon p) = (j_X(\iota(p)),0)$ for each $p \in P_1$. Note that this is well-defined because 
\begin{equation*}
\varepsilon \twid \iota(p) = \varepsilon\cdot (0,\iota(p)) = (j_X(\iota(p)),0) = \twid\iota(\varepsilon p).
\end{equation*}
Moreover, $\twid \iota$ is injective because $\im(\twid \iota_0) \intersect \varepsilon \twid D = (0)$. It is clear that $\twid \iota \bmod \varepsilon = \iota$ and $\im(\twid \iota)$ is saturated by Lemma \ref{lemma:saturated-deformation}.

We now prove part (b). A general $\twid \iota$ is determined by $\twid \iota(p)$ with $p \in P_1$ by $L[\varepsilon]$-linearity. Since $\twid \iota \congruent \iota \bmod \varepsilon$ we may write $\twid \iota(p) = (j_X(\beta(p)),\iota(p))$ where $\beta: P_1 \goto D$ is a $(\varphi,\Gamma_K)$-equivariant map. By assumption, $\beta$ is of the form $\mu \cdot \iota$ for some $\mu \in L$ and so $\twid \iota(p) = (1 + \mu\varepsilon)\cdot \twid \iota_0(p)$. This shows that $\twid \iota$ is unique up to a scalar $1 + \mu\varepsilon \in L[\varepsilon]^\times$.

It remain to show $h(X) = \coker(\twid \iota)$. By the previous paragraph, we may assume $\twid \iota$ is explicitly given as in the proof of (a). First note the map  $\twid D=f(X) \goto X$ given by $(x,d)\mapsto x$ is surjective. Furthermore, if $p,q \in P_1$ then $\twid\iota(p + q\varepsilon) = (j_X\iota(q),\iota(p)) \mapsto j_X\iota(q) \in (j_X\compose \iota)(P_1)$. Thus we have a well-defined $(\varphi,\Gamma_K)$-equivariant surjection $\alpha:\twid D/\im(\twid \iota) \surject X/(j_X\compose \iota)(P_1)$. Suppose that $\twid d \in \ker(\alpha)$ and write $\twid d = (j_X(\iota(p)),d)$ for some $p \in P_1$ and $d \in D$. Since $\theta(d) = \pi_X(j_X(\iota(p))) = 0$ we may write $d = \iota(p')$ for some $p' \in P_1$. But then $\twid d = \twid \iota(p' + \varepsilon p) \in \im(\twid \iota)$. This shows $\alpha$ is injective and so $h(X) = X/(j_X\compose \iota)(P_1) = \coker(\twid \iota)$.
\end{proof}

\begin{lemma}\label{lemma:composition-zero}
Let $D$ be a crystalline $(\varphi,\Gamma_K)$-module over $\cal R_{K,L}$ and let $\cal R_{K,L}(\delta_1) \ci D$ be a saturated rank one $(\varphi,\Gamma_K)$-submodule. Let $\phi_1$ be the crystalline eigenvalue appearing in $D_{\cris}(\cal R_{K,L}(\delta_1))$, assume that $\phi_1$ is simple in $D_{\crys}(D)$ and let $h_1=(h_{1,\tau})$ be the Hodge--Tate weights of $\delta_1$. Let $\fr t_D^{\phi_1,h_1}$ be the Zariski tangent space to $\fr X_{D}^{\phi_1,h_1} := \fr X_D^{\phi_1}\intersect \bigintersect_\tau \fr X_{D}^{h_{1,\tau}}$ as in Sections \ref{subsec:hodge-tate-deformations} and \ref{subsec:weakly-refined-deformations}.

\begin{enumerate}
\item The composition $\fr t_D^{\phi_1,h_1} \inject H^1(\ad D) \goto H^1_{/f}(\ad D) \goto H^1_{/f}(D(\delta_1^{-1}))$ is zero. 
\item If $Q_1 = \coker(\cal R_{K,L}(\delta_1) \goto D)$ and $H^2(D \tensor Q_1^\dual) = (0)$ then there is a natural embedding $\fr t_{D}^{\phi_1,h_1}/\fr t_{D,f} \inject H^1_{/f}(D\tensor Q_1^\dual)$.
%
\end{enumerate}
\end{lemma}
\begin{proof}
Suppose that $\twid D \in \fr t_{D}^{\phi_1,h_1}$. The least Hodge--Tate weight $h_{1,\tau}$ is constant in $\twid D$ for each embedding $\tau:K \inject L$. Thus there exists a $\twid \phi_1 \in L[\varepsilon]^\times$ deforming $\phi_1$ such that $D_{\cris}(\twid D)^{\varphi^{f_K} = \twid \phi_1}$ is free of rank one over $K_0\tensor_{\Q_p} L[\varepsilon]$ (compare with \eqref{eqn:HT-defs}). By \cite[Lemma 7.2]{Bergdall-ParabolineVariation}, the image of $\twid D$ under the natural map $\Ext^1_{(\varphi,\Gamma_K)}(D,D) \goto \Ext^1_{(\varphi,\Gamma_K)}(\cal R_{K,L}(\delta_1),D)$ lands inside the subspace $\Ext^1_f(\cal R_{K,L}(\delta_1),D)$ of crystalline extensions (the reference is valid because we have assumed that $\phi_1$ is a simple eigenvalue in $D_{\crys}(D)$). This proves part (a).

To prove part (b), we apply Lemma \ref{lemma:exactness-selmer-groups} to the short exact sequence
\begin{equation}\label{eqn:revising-seq}
0 \goto D \tensor Q_1^\dual \goto \ad D \goto D(\delta_1^{-1}) \goto 0
\end{equation}
of $(\varphi,\Gamma_K)$-modules. The hypotheses of Lemma \ref{lemma:exactness-selmer-groups} are satisfied because $D$, and thus each term in \eqref{eqn:revising-seq}, is crystalline and $H^2(D\tensor Q_1^\dual) = (0)$ by assumption in this lemma. We conclude from Lemma \ref{lemma:exactness-selmer-groups} that there is a natural short exact sequence
 \begin{equation}\label{eqn:application-lemma}
 0 \goto H^1_{/f}(D \tensor Q_1^\dual) \goto H^1_{/f}(\ad D) \goto H^1_{/f}(D(\delta_1^{-1})) \goto 0.
 \end{equation}
The subspace $\fr t_{D}^{\phi_1,h_1}/\fr t_{D,f} \ci H^1_{/f}(\ad D)$ maps to zero in the final term of \eqref{eqn:application-lemma} by part (a) of this lemma. This proves (b).
\end{proof}

\begin{lemma}\label{lemma:refined-hodge-tate-stable}
Let $D$ be a crystalline $(\varphi,\Gamma_K)$-module equipped with a regular generic triangulation $P_{\bullet}$ and let $\cal R_{K,L}(\delta_1) = P_1$. Set $Q_1 = \coker(\cal R_{K,L}(\delta_1) \goto D)$ and let $P'_{\bullet}$ be the induced triangulation on $Q_1$. The composition 
\begin{equation*}
\fr t_{D,P_{\bullet}}^{\Ref,\HT}/\fr t_{D,f} \inject H^1_{/f}(D\tensor Q_1^\dual) \goto H^1_{/f}(\ad Q_1)
\end{equation*}
has image inside the subspace $\fr t_{Q_1,P_{\bullet}'}^{\Ref,\HT}/\fr t_{Q_1,f} \ci H^1_{/f}(\ad Q_1)$.
\end{lemma}
\begin{remark}
There are minor clarifications needed for the lemma. First, the composition in Lemma \ref{lemma:refined-hodge-tate-stable} is well-defined by Lemma \ref{lemma:composition-zero}(b) and the inclusion $\fr t_{D,P_{\bullet}}^{\Ref,\HT} \ci \fr t_{D}^{\phi_1,h_1}$ (the notation as in Lemma \ref{lemma:composition-zero}). Second, the induced triangulation $P_{\bullet}'$ on $Q_1$ is regular generic because $P_{\bullet}$ itself is regular generic. This gives content to the conclusion of Lemma \ref{lemma:refined-hodge-tate-stable}.
\end{remark}
\begin{proof}[Proof of Lemma \ref{lemma:refined-hodge-tate-stable}]
First note that $H^2(\delta_j\delta_i^{-1}) = (0)$ for each $i,j$ because $P_{\bullet}$ is assumed to be regular generic. In particular, by the long exact sequence in cohomology we deduce that $H^2(D\tensor Q_1^\dual) = (0)$.

Write $\alpha$ for the composition $\fr t_{D,P_{\bullet}}^{\Ref,\HT}/\fr t_{D,f} \goto H^1_{/f}(D\tensor Q_1^\dual) \goto H^1_{/f}(\ad Q_1)$.  Now suppose that $[\twid D] \in \fr t_{D,P_{\bullet}}^{\Ref,\HT}/\fr t_{D,f}$. The tangent space $\fr t_{D,P_{\bullet}}^{\Ref,\HT}$ is contained in the tangent space $\fr t_D^{\phi_1,h_1}$ as in Lemma \ref{lemma:composition-zero}. Since $H^2(D\tensor Q_1^\dual) = (0)$, we may use Lemma \ref{lemma:composition-zero}(b) to choose a representative $\twid D \in \fr t_{D,P_{\bullet}}^{\Ref,\HT}$ such that $\twid D$ is in the image of $H^1(D\tensor Q^\dual_1) \goto H^1(\ad D)$.

By Lemma \ref{lemma:diagram-chase}(a) there exists a constant deformation $\twid\iota:\cal R_{K,L[\varepsilon]}(\delta_1) \inject \twid D$ which is a direct summand over $\cal R_{K,L[\varepsilon]}$. The space $H^0(D(\delta_1^{-1}))$ is one-dimensional since $P_{\bullet}$ is regular (see the remark following Definition \ref{defi:regular-generic}) and so Lemma \ref{lemma:diagram-chase}(b) implies that
\begin{equation*}
\alpha([\twid D]) = \coker(\twid \iota) \bmod H^1_f(\ad Q_1).
\end{equation*}
If $\twid Q_1 := \coker(\twid \iota) \in H^1(\ad Q_1)$ then $\twid Q_1$ has constant Hodge--Tate weights because $\twid D$ has constant Hodge--Tate weights. It suffices now to show that $\twid Q_1$ is a weakly-refined deformation with respect to the triangulation $P_{\bullet}'$.

Let $\twid \phi_1,\dotsc,\twid \phi_n$ be the elements of $L[\varepsilon]^\times$ witnessing $\twid D$ as being weakly-refined (see \eqref{eqn:HT-defs}). Note $\twid \phi_1$ must be equal to $\phi_1$ because $\phi_1$ is a simple eigenvalue in $D_{\cris}(D)$ and $\twid\iota$ witnesses $\phi_1$ appearing in $D_{\crys}(\twid D)$. Now consider $2 \leq m \leq n$ and the short exact sequence of $(\varphi,\Gamma_K)$-modules over $\cal R_{K,L[\varepsilon]}$
\begin{equation*}
0 \goto \wedge^{m-1}(\twid Q_1)\tensor_{\cal R_{K,L[\varepsilon]}}\cal R_{K,L[\varepsilon]}(\delta_1) \goto \wedge^m \twid D \goto \wedge^m \twid Q_1 \goto 0.
\end{equation*}
Recall that $D_{\cris}(-)$ is left exact. Since the eigenvalue $\phi_1\dotsc \phi_m$ appearing in $D_{\cris}(\wedge^m D)$ does not appear in $D_{\cris}(\wedge^m Q_1)$, by the regular condition on $P_{\bullet}$, we conclude that for $m=2,\dotsc,n$ there is a natural equality
\begin{equation*}
D_{\cris}\left(\wedge^{m-1}(\twid Q_1)\tensor_{\cal R_{K,L[\varepsilon]}} \cal R_{K,L[\varepsilon]} (\delta_1)\right)^{\varphi^{f_K} = \phi_1\twid \phi_2\dotsb \twid \phi_m}  \overto{\iso} D_{\cris}\left(\wedge^m (\twid D)\right)^{\varphi^{f_K} = \phi_1\twid \phi_2\dotsb \twid \phi_m}.
\end{equation*}
Twisting the left hand side by the constant crystalline  $(\varphi,\Gamma_K)$-module $\cal R_{K,L[\varepsilon]}(\delta_1^{-1})$ we have shown that $\twid Q_1 \in \fr t_{Q_1,P_{\bullet}'}^{\Ref}$.
\end{proof}
We now give an upper bound for the constant weight, weakly-refined, deformations, up to the crystalline deformations. We restate our hypotheses for clarity.
\begin{theorem}\label{theorem:HT-ref-bound}
If $D$ is a crystalline $(\varphi,\Gamma_K)$-module over $\cal R_{K,L}$ equipped with a regular generic triangulation $P_{\bullet}$ with parameter $(\delta_1,\dotsc,\delta_n)$ then 
\begin{equation}
\dim_L \fr t_{D,P_{\bullet}}^{\Ref,\HT}/\fr t_{D,f} \leq \sum_{1 \leq i < j \leq n} \dim_L H^1_{/f}(\delta_i\delta_j^{-1}).
\end{equation}
\end{theorem}
\begin{proof}
We argue by induction on $n$. If $n = 1$ then $\fr t_{D,P_{\bullet}}^{\Ref} = \fr t_{D}$ and $\fr t_{D,P_{\bullet}}^{\Ref,\HT} = \fr t_{D,f}$. Thus the bound is true (and an equality) in this case.

Now suppose that $n > 1$ and let $Q_1 = \coker(\cal R_{K,L}(\delta_1) \goto D)$. Then $H^2(Q_1^\dual(\delta_1)) = (0)$ because $P_{\bullet}$ is regular generic. By Lemmas \ref{lemma:exactness-selmer-groups} and \ref{lemma:composition-zero} we have a diagram with exact rows
\begin{equation*}
\xymatrix{
0 \ar[r] & U \ar[r] \ar[d] & \fr t_{D,P_{\bullet}}^{\Ref,\HT}/\fr t_{D,f} \ar[r]  \ar@{^{(}->}[d] & V \ar[d] \ar[r] & 0\\
0 \ar[r] & H^1_{/f}( Q_1^\dual(\delta_1)) \ar[r] & H^1_{/f}(D\tensor Q_1^\dual) \ar[r] & H^1_{/f}(\ad Q_1) \ar[r] & 0
}
\end{equation*}
where $U$ and $V$ are defined by the diagram itself. We separately bound $\dim_L U$ and $\dim_L V$.

By Lemma \ref{lemma:refined-hodge-tate-stable}, $V \ci \fr t_{Q_1,P_{\bullet}'}^{\Ref,\HT}/\fr t_{Q_1,f}$ where $P_{\bullet}'$ is the triangulation induced on $Q_1$ from $P_{\bullet}$. Thus by induction we have
\begin{equation*}
\dim_L V \leq \dim_L \fr t_{Q_1,P_{\bullet}'}^{\Ref,\HT}/\fr t_{Q_1,f} \leq \sum_{2 \leq i < j \leq n} \dim_L H^1_{/f}(\delta_i\delta_j^{-1}).
\end{equation*}
On the other hand, $U \ci H^1_{/f}(Q_1^\dual(\delta_1))$ and so
\begin{equation*}
\dim_L U \leq \dim_L H^1_{/f}(Q_1^\dual(\delta_1)) \overset{\text{Lemma \ref{lemma:exactness-selmer-groups}}}{=} \sum_{j=2}^n \dim_L H^1_{/f}(\delta_1\delta_j^{-1}). 
\end{equation*}
Putting the two upper bounds together we get the result.
\end{proof}
Finally we translate Theorem \ref{theorem:HT-ref-bound} into an upper bound in terms of critical types.
\begin{lemma}\label{lemma:characters}
If $\delta: K^\times \goto L^\times$ is a crystalline character and $H^2(\delta) = (0)$ then $\dim_L H^1_{/f}(\delta) = \sizeof\set{\tau: K \inject L \st \HT_\tau(\delta) \geq 0}$.
\end{lemma}
\begin{proof}
By the formula \eqref{eqn:selmer-dimension} we have that 
\begin{equation*}
\dim_L H^1_f(\delta) = \dim_L H^0(\delta) + \dim_L D_{\dR}(\cal R_{K,L}(\delta))/ D_{\dR}^+(\cal R_{K,L}(\delta)).
\end{equation*}
 On the other hand, since $H^2(\delta) = (0)$, the Euler--Poincar\'e characteristic formula \cite[Theorem 1.2(a)]{Liu-CohomologyDuality} implies that $\dim_L H^1(\delta) = (K:\Q_p) + \dim_L H^0(\delta)$. Thus,
\begin{align*}
\dim_L H^1_{/f}(\delta) &= (K:\Q_p) - \dim_L D_{\dR}(\cal R_{K,L}(\delta))/D_{\dR}^+(\cal R_{K,L}(\delta))\\
&= (K:\Q_p) - \sizeof\set{ \tau \st \HT_\tau(\delta) < 0}.
\end{align*}
The result is now clear.
\end{proof}
Recall that if $\sigma \in S_n$ then its length is given by $ \ell(\sigma) = \sizeof\set{(i,j) \st i < j \text{ and } \sigma(i) > \sigma(j)}$.
\begin{corollary}\label{corollary:upper-bound-HT}
Suppose that $D$ is a crystalline $(\varphi,\Gamma_K)$-module over $\cal R_{K,L}$ equipped with a regular generic triangulation $P_{\bullet}$ and let $(\sigma_\tau)_\tau$ be the critical type of $P_{\bullet}$. Then
\begin{equation*}
\dim_L \fr t_{D,P_{\bullet}}^{\Ref,\HT}/\fr t_{D,f} \leq \sum_\tau 
\ell(\sigma_\tau),
\end{equation*}
\end{corollary}
\begin{proof}
Write $(\delta_1,\dotsc,\delta_n)$ for the parameter of $P_{\bullet}$. By Theorem \ref{theorem:HT-ref-bound} we have 
\begin{equation*}
\dim_L \fr t_{D,P_{\bullet}}^{\Ref,\HT}/\fr t_{D,f} \leq \sum_{1 \leq i < j \leq n} \dim_L H^1_{/f}(\delta_i\delta_j^{-1}).
\end{equation*}
Since the Hodge--Tate weights of $D$ are distinct and the triangulation $P_{\bullet}$ is regular generic, Lemma \ref{lemma:characters} above implies that if $i\neq j$, then $\dim_L H^1_{/f}(\delta_i\delta_j^{-1}) = \sizeof\set{\tau \st \HT_\tau(\delta_i) > \HT_\tau(\delta_j)}$. Thus Theorem \ref{theorem:HT-ref-bound} gives us
\begin{align*}
\dim_L \fr t_{D,P_{\bullet}}^{\Ref,\HT}/\fr t_{D,f} \leq\sum_{1 \leq i < j \leq n} \dim_L H^1_{/f}(\delta_i\delta_j^{-1}) &= \sum_{\tau} \sizeof \set{(i,j) \st i < j \text{ and } \HT_\tau(\delta_i) > \HT_\tau(\delta_j)}\\
&= \sum_\tau \underlabel{\ell(\sigma_\tau)}{\sizeof \set{(i,j) \st i < j \text{ and } \sigma_\tau(i) > \sigma_\tau(j)}}.
\end{align*}
The final equality was the definition of the critical type (see Definition \ref{defi:critical-type}).
\end{proof}

\subsection{The relative tangent space for weakly-refined deformations}\label{subsec:rel-tangent-space}
Suppose that $D$ is a crystalline $(\varphi,\Gamma_K)$-module over $\cal R_{K,L}$ with distinct crystalline eigenvalues and distinct Hodge--Tate weights. We assume throughout this section that $D$ is also equipped with a regular generic triangulation $P_{\bullet}$ whose parameter we denote by $(\delta_1,\dotsc,\delta_n)$. We write $\fr t_{D,P_{\bullet}}^{\Ref}$ for the tangent space to the weakly-refined deformations with respect to $P_{\bullet}$.

If $\twid D \in \fr t_{D}$ then the uniqueness of the Hodge--Tate weights means $\twid D$ has Hodge--Sen--Tate weights $\set{\eta_{i,\tau}}_{i,\tau}$ in $L[\varepsilon]$ such that $\eta_{i,\tau} = h_{i,\tau} + \varepsilon\mathrm{d}\eta_{i,\tau}$, where $\mathrm{d}\eta_{i,\tau} \in L$. The map $\twid D \mapsto (\mathrm{d}\eta_{i,\tau})_{i,\tau}$ defines an $L$-linear map 
\begin{equation*}
\fr t_{D} \overto{\mathrm{d}\eta} \bigoplus_\tau L^{\dsum n}
\end{equation*}
and by definition of Hodge--Tate weights we have an exact sequence 
\begin{equation}\label{eqn:key-maps}
0 \goto \fr t_{D,P_{\bullet}}^{\Ref,\HT}/\fr t_{D,f} \goto \fr t_{D,P_{\bullet}}^{\Ref}/\fr t_{D,f} \goto \bigoplus_\tau L^{\dsum n}
\end{equation}
of $L$-vector spaces. 
\begin{lemma}\label{lemma:ramification-weights}
Let $(\sigma_\tau)_{\tau}$ be the critical type of $P_{\bullet}$. The image of $\mathrm{d}\eta$ is contained in the subspace $\bigoplus_\tau V_{\sigma_{\tau}} \ci \bigoplus_\tau L^{\dsum n}$ where 
\begin{equation*}
V_{\sigma_\tau} := \set{ (v_i) \in L^{\dsum n} \st v_{\sigma_\tau(i)} = v_{i} \text{ for $i=1,\dotsc,n$}}.
\end{equation*}
\end{lemma}
\begin{proof}
The lemma states $\eta_{\sigma_{\tau}(i),\tau} - \eta_{i,\tau}$ is constant for each $i,\tau$. This follows from \cite[Lemma 7.2]{Bergdall-ParabolineVariation} (compare with the proofs of Lemma \ref{lemma:composition-zero} and \cite[Theorem 7.1]{Bergdall-ParabolineVariation}).
\end{proof}

Recall that if $\sigma \in S_n$ then we write $c(\sigma)$ for the number of orbits in $\set{1,\dotsc,n}$ under the action of the cyclic group generated by $\sigma$.
\begin{theorem}\label{theorem:best-upper-bound}
Suppose that $D$ is a crystalline $(\varphi,\Gamma_K)$-module over $\cal R_{K,L}$ equipped with a regular generic triangulation $P_{\bullet}$ with critical type $(\sigma_{\tau})_{\tau}$. Then
\begin{equation*}
\dim \fr t_{D,P_{\bullet}}^{\Ref}/\fr t_{D,f} \leq \sum_{\tau} \ell(\sigma_\tau) +  c(\sigma_\tau).
\end{equation*}
\end{theorem}
\begin{proof}
It is easy to see that $\dim_L V_{\sigma_\tau} = c(\sigma_\tau)$ for each embedding $\tau$. By Lemma \ref{lemma:ramification-weights} and Corollary \ref{corollary:upper-bound-HT} we deduce that
\begin{equation*}
\dim_L \fr t_{D,P_{\bullet}}^{\Ref}/\fr t_{D,f} \leq \dim_L \fr t_{D,P_{\bullet}}^{\Ref,\HT}/\fr t_{D,f} + \sum_\tau c(\sigma_\tau) \leq \sum_\tau \ell(\sigma_\tau) + c(\sigma_\tau),
\end{equation*}
as we wanted.
\end{proof}
To end this section we briefly explain the upper bound we have produced. If $\sigma \in S_n$ then let $\ord(\sigma)$ be its order as an element of $S_n$. We leave the following lemma for the reader.
\begin{lemma}\label{lemma:cycle-lemma}
If $\sigma \in S_n$ is a cycle then $\ell(\sigma) + 1 \geq \ord(\sigma)$
with equality if and only if $\sigma$ is a product of distinct simple transpositions. 
\end{lemma}
\begin{proposition}\label{prop:permutations}
Let $\sigma \in S_n$. Then $\ell(\sigma) + c(\sigma) \geq n$ with equality if and only if $\sigma$ is a product of distinct simple transpositions.
\end{proposition}
\begin{proof}
Write $\sigma = \sigma_1\dotsb \sigma_r$ where the $\sigma_i$ are {\em disjoint} cycles. In particular, $\sigma$ is a product of distinct simple transpositions if and only if each $\sigma_i$ is. Next, we note that \begin{itemize}
\item $\ell(\sigma) = \sum \ell(\sigma_i)$ and
\item $c(\sigma) = n + r - \sum \ord(\sigma_i)$.
\end{itemize}
Thus we have
\begin{equation*}
\ell(\sigma) + c(\sigma) = n + \sum_{i=1}^r \left(\ell(\sigma_i) + 1 - \ord(\sigma_i)\right).
\end{equation*}
Lemma \ref{lemma:cycle-lemma} implies that the terms in the sum are all non-negative. This shows $\ell(\sigma) + c(\sigma) \geq n$ always. Moreover, we have equality if and only if $\ell(\sigma_i) + 1 - \ord(\sigma_i) = 0$ for all $i=1,\dotsc,r$. But Lemma \ref{lemma:cycle-lemma} also implies that this is equivalent to $\sigma_i$ being a product of distinct simple tranpositions for each $i=1,\dotsc,r$, so we are done.
\end{proof}

\begin{corollary}\label{corollary:bound}
Suppose that $D$ is a crystalline $(\varphi,\Gamma_K)$-module over $\cal R_{K,L}$ equipped with a regular generic triangulation $P_{\bullet}$ with critical type $(\sigma_{\tau})_{\tau}$. If each $\sigma_{\tau}$ is a product of distinct simple transpositions then $\dim_L \fr t_{D,P_{\bullet}}^{\Ref}/\fr t_{D,f} \leq (K:\Q_p)\cdot n$.
\end{corollary}
\begin{proof}
Combine Proposition \ref{prop:permutations} and Theorem  \ref{theorem:best-upper-bound}.
\end{proof}

\begin{remark}
Ostroff showed the author an easy argument that the number of $\sigma \in S_n$ which are products of distinct simple transpositions is given by $F_{2n}$ where $F_m$ is the Fibonacci sequence starting with $F_1 = 0$ and $F_2 = 1$. In particular, the proportion $F_{2n}/n!$ of permutations which are products of distinct simple transpositions tends to zero as $n \goto +\infty$.
\end{remark}

\section{Application to eigenvarieties}\label{section:application}

\subsection{Unitary groups and Galois representations}
Our goal in this subsection is to specify notations and conventions for unitary groups, automorphic representations, and Galois representations. We do not strive for the greatest generality; our goal is to illustrate how the local deformation calculation in Section \ref{section:weakly-refined} can be used to bound dimensions of tangent spaces on eigenvarieties. Our hypotheses may be weakened in various directions, especially as progress is made in constructing Galois representations and Langlands functoriality.

Let $F/F^+$ be a CM extension of number fields with $F^+$ totally real and $F$ a totally imaginary quadratic extension of $F^+$. We assume $F/F^+$ is unramified everywhere and each $p$-adic place of $F^+$ splits in $F$.

Fix an integer $n\geq 1$ which is either odd or if $n$ is even then assume that $n(F^+:\Q) \congruent 0 \bmod 4$. With this, we let $\bf G$ denote a unitary group in $n$ variables over $F^+$ such that
\begin{itemize}
\item $\bf G$ is split over $F$. We fix an isomorphism $\bf G\times_{F^+} F \simeq {\GL_n}_{/F}$.
\item $\bf G$ is quasi-split at each finite place of $F^+$.
\item $\bf G(F^+\tensor_{\Q} \R)$ is a finite product of copies of the compact real unitary group $\Un(n)$.
\end{itemize}
We refer to $\bf G$ as a definite unitary group associated to $F/F^+$.

If $w$ is a place of $F^+$ then let $F^+_w$ denote the corresponding local field (and similarly for places of $F$). If $w$ splits in $F$ then the choice of $\twid w \dvd w$ determines an isomorphism $\bf G(F_{w}^+)\simeq \GL_n(F_{\twid w}) = \GL_n(F_w^+)$ which implicitly depends on $\twid w$.

We fix a compact open subgroup $U^p = \prod_{w \ndvd p} U_w \ci \bf G(\A_{F^+}^{p\infty})$. Here, $\Af_{F^+}^{p\infty}$ is the finite adeles of $F^+$ away from $p$. We assume throughout that $U_w$ is maximal hyperspecial compact at every inert place of $F^+$. Finally, we also write $S$ for a finite set of finite places of $F^+$ such that $S$ contains all the $p$-adic places and all the places $w$ such that $U_w$ is not maximal hyperspecial compact. In particular, each place in $S$ is split in $F$. We will also use $S$ for the set of places $\twid w$ of $F$ such that $\twid w \dvd w$ with $w \in S$.

Choose an isomorphism $\C \simeq \bar \Q_p$. Then, each complex embedding $v_{\infty} : F^+ \inject \C$ corresponds uniquely to a pair $(v,\tau)$ where $v$ is a $p$-adic place of $F^+$ and $\tau: F_v^+\inject \bar \Q_p$ is an embedding. Given the pair $(v,\tau)$ we write $(v,\tau)_\infty$ for the corresponding infinite place of $F^+$.

Automorphic representations are irreducible direct summands in the space of complex-valued functions on $\G(F^+)\leftmod \G(\Af_{F^+})$ which are smooth and $G(F^+\tensor_{\Q} \R)$-finite (see \cite[Section 6.2.3]{BellaicheChenevier-Book} for example). Every automorphic representation $\pi$ may be factored as $\pi = \pi_\infty \tensor \pi_f$, where $\pi_f = \bigtensor' \pi_w$ is a representation of $\bf G(\Af_{F^+}^\infty)$ and $\pi_\infty$ is the weight of $\pi$, which is an algebraic representation of $\bf G(F^+\tensor_{\Q} \R)$. The weight $\pi_\infty$ factors as a tensor product $\bigtensor_{v_{\infty}} \pi_{v_\infty}$ of irreducible algebraic representations of $\Un(n)$ indexed by infinite places $v_{\infty}$ of $F^+$. Let $k_{v_\infty} = (k_{1,v_\infty}\geq k_{2,v_\infty}\geq \dotsb \geq k_{n,v_\infty})$ be the dominant weight associated to $\pi_{v_\infty}$. If $(v,\tau)$ is a pair as in the previous paragraph we write $k_{i,v,\tau} := k_{i,(v,\tau)_\infty}$ and for fixed $v,\tau$ we stress that $k_{1,v,\tau} \geq k_{2,v,\tau} \geq \dotsb \geq k_{n,v,\tau}$. We say that $\pi$ has tame level $U^p$ if $\pi_f^{U^p}\neq 0$.

By the work of many authors, for each automorphic representation $\pi$ of tame level $U^p$ there is a unique $n$-dimensional continuous semi-simple representation $\rho_{\pi} : G_{F,S} \goto \GL_n(\bar \Q_p)$ such that:\footnote{The notations and conventions may be found in Section \ref{subsec:notations}. Below, if $\twid w$ is a place of $F$, we write $\rho_{\pi,\twid w}$ for the restriction of $\rho_{\pi}$ to a decomposition group of $\twid w$. Everything depends the isomorphism $\C \simeq \Q_p$.}
\begin{enumerate}[(LCG-I)]
\item$\rho_{\pi}$ conjugate self-dual up to a twist, i.e. $\rho^{\perp}_\pi \simeq \rho_\pi(n-1)$ where $\rho^\perp$ is the conjugate dual representation $g \mapsto {}^t\rho(\twid cg\twid c)^{-1}$ ($\twid c \in G_{F^+}$ is any order two lift of the non-trivial element in $\Gal(F/F^+)$; see \cite{BellaicheChenevier-Sign}).\label{lcg-infinite}
\item If $w \ndvd p$ is a place of $F^+$ and $\twid w \dvd w$ is a place of $F$ then $\WD(\rho_{\pi,\twid w}) = \rec(\pi_w^{\twid w}\abs{\det}^{1-n\over 2})$, where $\pi_w^{\twid w}$ is either the irreducible smooth representation of $\GL_n(F_{\twid w})$ associated to $\pi_w$ via the isomorphism $\mathbf G(F_w^+)\simeq \GL_n(F_{\twid w})$ if $w$ is split in $F$, or $\pi_w^{\twid w}$ is the base change (see \cite{Minguez-UnramifiedUnitaryGroups}, for example) from $G(F_w^+)$ to $\GL_n(F_{\twid w})$ of the necessarily unramified representation $\pi_w$ of $G(F_w^+)$ if $w$ is inert (and thus $w \nin S$).

 \label{lcg-notp}
\item If $v$ is a $p$-adic place of $F^+$ and $\twid v \dvd v$ is a place of $F$ then the representation $\rho_{\pi,\twid v}$ is potentially semi-stable and $\WD(\rho_{\pi,\twid v}) = \rec(\pi_v^{\twid v}\abs{\det}^{1-n\over 2})$ (with the notation  as above). If $\pi_v$ is unramified then $\rho_{\pi,\twid v}$ is crystalline.

The Hodge--Tate weights are as follows. If $\tau : F_{\twid v} \inject \bar \Q_p$ is an embedding then the natural equality $F_{v}^+ = F_{\twid v}$ defines a pair $(v,\tau)$ of a $p$-adic place $v$ of $F^+$ together with an embedding $\tau:F_v^+ \inject \bar \Q_p$. Then, the $\tau$-Hodge--Tate weights of $\rho_{\pi,\twid v}$ are given by $h_{i,\twid v,\tau} = - k_{i,v,\tau} + i-1$.\label{lcg-p}
\end{enumerate}
Generally, the representations $\rho_{\pi}$ are constructed in two steps. The first, requiring that $\bf G$ is quasi-split at each finite place, is to apply the base change theorems of Labesse to \cite{Labesse-BaseChangeUnitary} to transfer to automorphic representations for ${\GL_n}_{/F}$. (We used that $F/F^+$ is everywhere unramified to not have to address ramified primes in (LCG-\ref{lcg-notp}).)

The second step is the vast collection of works on constructing Galois representations for regular algebraic essentially conjugate self-dual representations of ${\GL_n}_{/F}$ along with their local properties. See \cite{Shin-ConstructionGaloisReps, ChenevierHarris-Construction} for further references (along with \cite{BLGGT-LGClequalp2, Caraiani-LGC,Caraiani-MonodromyLGC} for the various compatibilities, especially the compatibility at $w \in S$ with $w \ndvd p$ given by \cite{Caraiani-LGC}).

\subsection{Refinements and eigenvarieties}\label{subsec:eigenvarieties}
An eigenvariety $p$-adically interpolates automorphic representations $\pi$ for $\bf G$, together with triangulations of the corresponding crystalline Galois representation (or orderings of crystalline eigenvalues). 

Let us be more precise. Write $\scr H(U^p)^{\sph}$ for the spherical Hecke algebra of tame level $U^p$. For the next two paragraphs, fix an automorphic representation $\pi$ for $\bf G$ of tame level $U^p$. Then, $\pi$ naturally gives rise to a ring homomorphism $\lambda_{\pi} : \scr H(U^p)^{\sph} \goto \bar \Q_p$. By (LCG-\ref{lcg-notp}) and the Cebotarev density theorem, $\rho_{\pi}$ is determined by $\lambda_{\pi}$. 

For each $p$-adic place $v$ of $F^+$ we now {\em choose} a distinguished place $\twid v \dvd v$ in $F$.\footnote{This choice is ultimately inconsequential. See the remark preceding Proposition \ref{prop:galois-surjection}.} Thus we have an identification $\bf G(F_v^+) = \GL_n(F_v^+)$ that remains set throughout the rest of this section. We let $T(F_v^+)$, respectively $B(F_v^+)$,  denote the subgroups of $\bf G(F_v^+)$ corresponding to the diagonal matrices, respectively the Borel subgroup of upper triangular matrices, and also the actual subgroups of $\GL_n(F_v^+)$.

If $v$ is a $p$-adic place of $F^+$ and $\pi_v$ is unramified then we may also {\em choose} a smooth unramified character $\vartheta_{\twid v} : T(F_v^+) \goto \bar \Q_p^\times$ such that $\pi_v \inject \Ind_{B(F_v^+)}^{\GL_n(F_v^+)}(\delta_{B(F_v^+)}^{1/2}\vartheta_{\twid v})$ (recall the notations from Section \ref{subsec:notations}). Note that $\vartheta_{\twid v}$ actually depends on $\twid v$, as the identification of $\GL_n(F_v^+)$ with $\G(F_v^+)$ does. Following \cite[Chapter 6]{BellaicheChenevier-Book}, we call $\vartheta_{\twid v}$ an accessible refinement for $\pi_v$. We also denote by $\psi_{\twid v}: T(F_v^+) \goto \bar \Q_p^\times$ the highest weight of the unique irreducible algebraic representation $\Res_{F_v^+/\Q_p} \GL_n$ over $\bar \Q_p$ with  weight given by the tuple $(k_{i,v,\tau})_{i,\tau}$. Explicitly, if $z \in (F_v^+)^\times$ then $\psi_{\twid v}(\diag(1,1,\dotsc,z,\dotsc,1,1)) = \prod_\tau \tau(z)^{k_{i,v,\tau}}$ where $z$ appears in the $i$th spot on the left-hand side of that equation.

Now let $T = \prod_{v} T(F_v^+)$ and $\hat T$ be the rigid analytic space over $\Q_p$ parameterizing continuous character of $T$. Then, for each automorphic representation $\pi$ of tame level $U^p$ which is unramified at the $p$-adic places we have a point $\chi_{\pi} :=  (\delta_{B(F_v^+)}^{-1/2}\psi_{\twid v}\vartheta_{\twid v})_{v \dvd p}$ in $\hat T$, depending on the choice of accessible refinements $\vartheta_{\twid v}$. The eigenvariety $X_{U^p}$ of tame level $U^p$ is coarsely defined as the rigid analytic closure of the points
\begin{equation*}
Z_{\cl} := \set{(\lambda_{\pi},\chi_{\pi})} \ci \Hom(\scr H(U^p)^{\sph},\bar \Q_p) \times \hat T
\end{equation*}
where the collection runs over automorphic representations $\pi$ of tame level $U^p$, unramified at the $p$-adic places, together with the choice of an accessible refinement at each $p$-adic places. We call $Z_{\cl}$ the set of ``classical points". The rigid analytic closure does not literally make sense, but we refer to \cite{Chenevier-pAdicAutomorphicForm, Emerton-InterpolationEigenvariety} for details and precisions on the construction. For the remainder of this section, we summarize the properties that we will need. 

The natural map $\chi: X_{U^p} \goto \hat T$ is written $x \mapsto \chi_x$. By way of comparison with other sources, the map $\chi$ contains the data that may usually be included in the presence of the Atkin--Lehner algebra (in the style of \cite{Chenevier-pAdicAutomorphicForm,BellaicheChenevier-Book} for example).

We briefly observe the role of our normalizations. If $(\lambda_\pi,\chi_\pi)$ is a classical point, and if $v \dvd p$ is a $p$-adic place we write $\chi_{1,v,\pi} \otimes \dotsb \otimes \chi_{n,v,\pi} : T(F_v^+) \goto \bar \Q_p^\times$ for the local component at $v$ of $\chi_\pi$. Then, (LCG-\ref{lcg-p}) implies that the crystalline eigenvalues of $\rho_{\pi,\twid v}$ are given by the list
\begin{equation}\label{eqn:crystalline-eigenvalues}
\left(\chi_{1,v,\pi}(\varpi_v)\cdot \prod_{\tau} \tau(\varpi_v)^{h_{1,v,\tau}}, \dotsc, \chi_{n,v,\pi}(\varpi_v)\cdot \prod_{\tau} \tau(\varpi_v)^{h_{n,v,\tau}}\right)
\end{equation}
for some/any choice of uniformizer $\varpi_v \in F_v^+$. In particular, each classical point (with distinct crystalline eigenvalues) is naturally equipped with triangulations at the $p$-adic places.
  
The classical points $Z_{\cl}$ are Zariski dense and accumulating in $X_{U^p}$ \cite[Theorem 7.3.1(v)]{BellaicheChenevier-Book}. The $p$-adic analytic variation of pseudocharacters may be used to construct a global pseudocharacter $T: G_{F,S} \goto \GL_n(\cal O(X_{U^p}))$ which interpolates $x \mapsto \tr(\rho_x)$ at classical points \cite[Proposition 7.1.1]{Chenevier-pAdicAutomorphicForm}. Specializing $T$ to a point $x$, \cite[Theorem 1(2)]{Taylor-Pseudoreps} also gives a continuous semi-simple representation
\begin{equation*}
\rho_x: G_{F,S} \goto \GL_n(\bar \Q_p)
\end{equation*}
which is conjugate self-dual and satisfies apparent compatibility over $X_{U^p}$ at unramified places (by interpolation). If $x \in X$ we let $L(x)$ denote its residue field.
\begin{lemma}\label{lemma:irreducible-globalize}
If $x \in X_{U^p}$ and $\rho_{x}$ is absolutely irreducible and defined over $L(x)$ then there exists an affinoid neighborhood $Y = \Sp(B) \ci X$ and a continuous representation $\rho_Y: G_{F,S} \goto \GL_n(B)$ such that $\rho_Y\tensor_{B} L(y) = \rho_y$ for all $y \in Y$.
\end{lemma}
\begin{proof}
Let $A$ be the rigid local ring $A = {\cal O}^{\rig}_{X,x}$. Since $A$ is Henselian (\cite[Theorem 2.1.5]{Bekovich-EtaleCohomology}) and $\rho_x$ is absolutely irreducible, $\rho_x$ lifts uniquely to a continuous representation $\rho_A:G_{F,S} \goto \GL_n(A)$ \cite[Corollarie 5.2]{Rouqier-Jalgebra96-Pseudocharacters}. The existence of $Y = \Sp(B)$ and $\rho_Y$ such that $\tr(\rho_y) = \tr(\rho_Y\tensor_{B} L(y))$ for all $y \in Y$ follows from \cite[Lemma 4.3.7]{BellaicheChenevier-Book}. But, the locus of points $y \in Y$ for which $\rho_y$ is absolutely irreducible is a rigid open subspace (\cite[Section 7.2.1]{Chenevier-pAdicAutomorphicForm}), so we can shrink $Y$ and assume that $\rho_y$ is absolutely irreducible for all $y$. In that case the equality $\tr(\rho_y)  = \tr(\rho_Y\tensor_B L(y))$ implies that $\rho_y = \rho_Y\tensor_B L(y)$, finishing the proof.
\end{proof}
\begin{remark}
The representations $\rho_x$ for classical $x$ satisfy the hypothesis of the lemma, i.e. they are defined over their residue fields.
\end{remark}

Instead of working with the whole tame level $U^p$ eigenvariety, we will instead consider a minimal eigenvariety. We briefly explain, but see \cite[Example 7.5.1]{BellaicheChenevier-Book} for further information (see also \cite[Section 3.6]{Chenevier-Eigenvariety}). Fix an automorphic representation $\pi$, unramified at the $p$-adic places, and the choice of an accessible refinement giving rise to a point $x_{\pi} \in X_{U^p}$.  The places $w \in S$ with $w \ndvd p$ are all split. Fix a choice $\twid w \dvd w$ for each such $w$. Then the representation $\pi_w^{\twid w}$ of $\G(F_w^+) \simeq \GL_n(F_w^+)$ has a $K$-type (\cite[Section 6.5]{BellaicheChenevier-Book}), which gives rise to idempotents  $e_w$ (independent of $\twid w$) commuting with $\scr H(U^p)^{\sph}$ inside the space of compactly supported continuous complex-valued functions $\scr C_c^0(\G(\Af_{F^+}^p))$. The minimal eigenvariety $X$ for $x = x_{\pi}$ is then the idempotent-type eigenvariety (\cite[Section 7.3]{BellaicheChenevier-Book}) obtained from the idempotents $(e_w)_{w \in S, w \ndvd p}$. This defines a closed rigid subvariety $X \inject X_{U^p}$. The corresponding classical points are those $(\lambda_{\pi},\chi_{\pi}) \in Z_{\cl}$ above such that $e_w(\pi)\neq 0$ for $w \in S$, $w \ndvd p$. 

We will need a property of $X$ relating to Galois representations at the ramified places $S$. Let $w \in S$ and $\twid w$ be a place of $F$ above $w$. If $(r,N)$ and $(r',N')$ are two Weil--Deligne representations of the Weil group $W_{F_{\twid w}}$ then we will use the notation $N \prec_{\twid w} N'$ for the ``less monodromy'' notation $N\prec_{I_{\twid w}} N'$ introduced in \cite[Definition 7.8.18]{BellaicheChenevier-Book} and $N\sim_{\twid w} N'$ for the obvious equal version. We note two things: 
\begin{enumerate}[(i)]
\item If $N\sim_{\twid w} N'$ then $\restrict{r}{I_{\twid w}}\simeq \restrict{r'}{I_{\twid w}}$. This is by definition of the relation $\prec$.
\item Let $x \in X_{U^p}$ be classical and $X$ be its minimal eigenvariety. If $z$ is a classical point on $X$ then $N_{z,\twid w} \prec_{\twid w} N_{x,\twid w}$ for all $w \in S$, where $(r_{z,\twid w},N_{z,\twid w}) = \WD(\rho_{z,\twid w})$ is the Weil--Deligne representation associated to $\rho_{z,\twid w}$. This follows from (LCG-\ref{lcg-notp}) and the definition of the idempotents $e_{w}$ (compare with \cite[Section 6.5]{BellaicheChenevier-Book}).
\end{enumerate}
We now summarize the rest of the properties of the minimal eigenvariety:

\begin{proposition}\label{proposition:minimal-eigenvarieties-info}
Let $x \in X_{U^p}$ be a classical point and $X$ its minimal eigenvariety. 
\begin{enumerate}
\item $X$ is equidimensional of dimension $(F^+:\Q)\cdot n$.
\item If $\rho_x$ is absolutely irreducible then there exists a canonical lifting $\hat \rho_x$ to $\cal O_{X,x}^{\rig}$, and for each $\twid w \dvd w$ with $w \in S$ and $w \ndvd p$,  we have $\restrict{\hat{\rho}_x}{I_{\twid w}} \simeq \restrict{\rho_x}{I_{\twid w}} \tensor_{L(x)} {\cal O}_{X,x}^{\rig}$.
\item If $\rho_x$ is absolutely irreducible, $\hat \rho_x$ is as in part (b), and $\scr H_x$ denotes the image of $\scr H(U^p)^{\sph}$ in the local ring $\cal O_{X,x}^{\rig}$ then $\scr H_x$ is contained in the the sub-algebra of $\cal O_{X,x}^{\rig}$ generated by $\tr \wedge^i\hat \rho_{x}(G_{F,S})$ for $i=1,2,\dotsc,n$.
\end{enumerate}
\end{proposition}
\begin{proof}
The statement (a) is a general property of eigenvarieties of idempotent-type attached to definite unitary groups \cite[Theorem 7.3.1(a)]{BellaicheChenevier-Book}. 

The lifting in (b) is deduced from Lemma \ref{lemma:irreducible-globalize}. The constancy of inertia acting is deduced from \cite[Corollary 7.5.10]{BellaicheChenevier-Book}. This is explained in the proof of  \cite[Proposition 7.6.10]{BellaicheChenevier-Book}, but since it is where we use the minimal eigenvariety (specifically comment (ii) above), we include the argument for convenience. Let $\cal K_x$ be the total ring of fractions of $\cal O_{X,x}^{\rig}$ and write $\cal K_x = \prod \cal K_{s(x)}$ where each $\cal K_{s(x)}$ is a field corresponding to an irreducible component $s(x)$. Let $\rho_x^{\gen}: G_{F,S} \goto \GL_n(\cal K_{s(x)})$ be the corresponding Galois representation and $\rho_{s(x),\twid w}^{\gen}$ be its restriction to the local group $G_{F_{\twid w}}$.  This representation admits a Weil--Deligne representation $(r_{s(x),\twid w}^{\gen},N_{s(x),\twid w}^{\gen})$ (see \cite[Section 7.8.4]{BellaicheChenevier-Book}). In general, $N_{z,\twid w} \prec_{\twid w} N_{s(x),\twid w}^{\gen}$ for all $z$ on $s(x)$ with $\sim_{\twid w}$ on a Zariski-dense subset. For instance, $N_{z,\twid w} \sim_{\twid w} N_{s(x),\twid w}^{\gen}$ for a set of classical $z$ on $s(x)$ accumulating at $x$. Thus (ii) above implies $N_{s(x),\twid w}^{\gen} \prec_{\twid w} N_{x,\twid w}$. The reverse is always true, so we conclude $N_{s(x),\twid w}^{\gen} \sim_{\twid w} N_{x,\twid w}$. Since $s(x)$ is arbitrary, we see that the hypothesis of \cite[Corollary 7.5.10]{BellaicheChenevier-Book} is satisfied.\footnote{The embedded reference to Proposition 7.5.8 in {\em loc. cit} is valid because we've assume that each place of $S$ splits in $F$.} We conclude that $\hat N_{x,\twid w} \sim_{\twid w} N_{x,\twid w}$ where $(\hat r_{x,\twid w}, \hat N_{x,\twid w})$ is the Weil--Deligne representation associated to $\hat \rho_{x,\twid w}$ (see the references just prior to \cite[Corollary 7.5.10]{BellaicheChenevier-Book}). By definition then $\restrict{\hat \rho_x}{I_{F_{\twid w}}} \simeq \restrict{\rho_x}{I_{F_{\twid w}}} \tensor_{L(x)} \cal O_{X,x}^{\rig}$ (see (i) above).

Part (c) follows from applying (LCG-\ref{lcg-notp}) to the places $w \nin S$. Indeed, the Satake isomorphism and (LCG-\ref{lcg-notp}) implies that $\scr H(U^p)^{\sph}$ is generated as an algebra by elements whose specialization at any classical point $z$ are the coefficients of the characteristic polynomial of $\rho_{z,\twid w}(\Frob_{\twid w})$. In particular, the image $\scr H_x$ in $\cal O_{X,x}^{\rig}$ is naturally generated by the characteristic polynomials of $\hat \rho_x(\Frob_w)$ with $w \nin S$.
\end{proof}

\subsection{Upper bounds for tangent spaces via deformation theory}\label{subsec:ub}
We continue to use the notations and conventions of the previous two sections. We also assume:
\begin{enumerate}[(\ref{subsec:ub}-A)]
\item  $\pi$ is an automorphic representation of tame level $U^p$, unramified above $p$, and $(\vartheta_{\twid v})_{\twid v}$ is a list of accessible refinements for the representations $\pi_v$ at the $p$-adic places $v$ of $F^+$. \label{ass-point}
\item The global Galois representation $\rho_{\pi}$ is absolutely irreducible. \label{ass-irred}
\item By (LCG-\ref{lcg-p}) the choices in (\ref{subsec:ub}-\ref{ass-point}) define orderings of the crystalline eigenvalues for each $\rho_{\pi,\twid v}$ (see \eqref{eqn:crystalline-eigenvalues}). We {\em assume} that these eigenvalues are all distinct and denote by $P_{\twid v,\bullet}$ the corresponding triangulation of the $(\varphi,\Gamma_{F_v^+})$-module $D_{\rig}^{\dagger}(\rho_{\pi,\twid v})$. \label{ass-distinct}
\item We {\em further} assume that $P_{\twid v,\bullet}$ is regular generic for each $v \dvd p$. \label{ass-generic}
\end{enumerate}
We now denote by $X \subset X_{U^p}$ the minimal eigenvariety containing the point $x = x_{\pi}$. By assumption (\ref{subsec:ub}-\ref{ass-irred}), $\rho_{x}$ is absolutely irreducible and so the  universal deformation functor $\fr X_{\rho_{x}}$ on $\fr{AR}_{L(x)}$ is (pro-)representable by a complete local noetherian $L(x)$-algebra $R_{\rho_{x}}^{\univ}$ with residue field $L(x)$. Denote by $\fr X_{\rho_{x}}^{\oper{csd}} \ci \fr X_{\rho_x}$ the representable subfunctor parameterizing conjugate self-dual deformations $\rho$ (i.e. $\rho^{\perp} \simeq \rho(n-1)$).

At places $\twid w \ndvd p$ of $F$, we may consider the local Galois representation $\rho_{x,\twid w}$ and its universal deformation functor $\fr X_{\rho_{x,\twid w}}$. There is a natural subfunctor $\fr X_{\rho_{x,\twid w},f} \ci \fr X_{\rho_{x,\twid w}}$ parameterizing deformations which are minimally ramified. That is, if $A \in \fr{AR}_{L(x)}$ then
\begin{equation*}
\fr X_{\rho_{x,\twid w},f}(A) = \set{\rho_A \in \fr X_{\rho_{x,\twid w}}(A) \st \restrict{\rho_A}{I_{\twid w}} \simeq \restrict{\rho_{x,\twid w}}{I_{\twid w}}\tensor_{L(x)} A}.
\end{equation*}
For example, if $\twid w \nin S$ then $\rho_A \in \fr X_{\rho_{x,\twid w},f}(A)$ if and only if $\rho_A$ is trivial on inertia at $\twid w$.
The arrow $\fr X_{\rho_{x,\twid w},f} \inject \fr X_{\rho_{x,\twid w}}$ is relatively representable (e.g. by Schlessinger's criterion). 

For $v \dvd p$ in $F^+$, we let $\fr X_{\rho_{x,\twid v}}^{\Ref} := \fr X_{D_{\rig}^{\dagger}(\rho_{x,\twid v}), P_{\twid v,\bullet}}^{\Ref}$ be the weakly-refined deformations of $D_{\rig}^{\dagger}(\rho_{x,\twid v})$ with respect to the triangulation $P_{\twid v,\bullet}$ arising from (\ref{subsec:ub}-\ref{ass-distinct}). We now let $\fr X_{\rho_{x}}^{\Ref,\min}$ denote the fibered product defined by the diagram
\begin{equation*}
\xymatrix{
\fr X_{\rho_{x}}^{\Ref,\min} \ar[d] \ar[r] & \fr X_{\rho_{x}}^{\oper{csd}} \ar[d] \\
 \prod_{\twid w \ndvd p} \fr X_{\rho_{\pi,\twid w},f} \times \prod_{v \dvd p} \fr X_{\rho_{x,\twid v}}^{\Ref} \ar[r] & \prod_{\twid w \ndvd p} \fr X_{\rho_{x,\twid w}} \times \prod_{v \dvd p} \fr X_{\rho_{x,\twid v}}.
}
\end{equation*}
Since the bottom arrow is relatively representable and $\fr X_{\rho_{x}}^{\oper{csd}}$ is itself (pro-)representable, we deduce that $\fr X_{\rho_{x}}^{\Ref,\min}$ is a (pro-)representable functor on $\fr{AR}_{L(x)}$. We denote the universal ring representing $\fr X_{\rho_x}^{\Ref,\min}$ by $R_{\rho_{x}}^{\Ref,\min}$.
\begin{remark}
The conjugate self-dual property of the deformations in $\fr X_{\rho_x}^{\Ref,\min}$ implies that it is sufficient to consider only one place of $F$ above each place of $F^+$. This is why, for example, we only specify the deformation problem at the fixed $p$-adic places $\twid v$ of $F$.
\end{remark}

Let $\cal O_{X,x}^{\rig}$ denote the rigid analytic local ring of $x$ at $X$ and $\hat{\cal O}_{X,x}^{\rig}$ denote its completion.
\begin{proposition}\label{prop:galois-surjection}
There is a canonical surjective ring homomorphism $f: R_{\rho_{x}}^{\Ref,\min} \surject \hat{\cal O}_{X,x}^{\rig}$.
\end{proposition}
\begin{proof}
Since $\rho_x$ is absolutely irreducible, Proposition \ref{proposition:minimal-eigenvarieties-info}(b) implies that we can canonically lift $\rho_x$ to a deformation $\hat{\rho}_x: G_{F,S} \goto \GL_n(\cal O_{X,x}^{\rig}) \subset \GL_n(\hat{\cal O}_{X,x}^{\rig})$. Since $\hat{\cal O}_{X,x}^{\rig}$ is a complete local Noetherian $L(x)$-algebra with residue field $L(x)$, this defines a map $f: R_{\rho_{x}}^{\univ} \goto \hat{\cal O}_{X,x}^{\rig}$. To show that $f$ factors through $R_{\rho_x}^{\Ref,\min}$, we need to show that $\hat \rho_x$ defines a point in the subfunctor $\fr X_{\rho_x}^{\Ref,\min} \ci \fr X_{\rho_x}$.  For that, we need to check:
\begin{enumerate}[(i)]
\item $\hat \rho_x$ is conjugate self-dual.
\item $\hat \rho_x$ is minimally ramified away from $p$.
\item If $v \dvd p$ then the representation $\hat \rho_{x,\twid v}$ is a weakly-refined deformation of $\rho_{x,\twid v}$.
\end{enumerate}
The first point is clear because the Galois representation $\rho_{y}$ is conjugate self-dual for all $y \in X$. The condition (ii) follows from Proposition \ref{proposition:minimal-eigenvarieties-info}(b), which is valid by assumption (\ref{subsec:ub}-\ref{ass-irred}) and the minimality of $X$. The point (iii) is the crucial $p$-adic interpolation of crystalline eigenvalues over eigenvarieties as we now explain  (see \cite{Kisin-OverconvergentModularForms, BellaicheChenevier-Book, Liu-Triangulations,KedlayaPottharstXiao-Finiteness}).

Let $Y$ be an affinoid open of $X$, containing $x$, as in Lemma \ref{lemma:irreducible-globalize}. Fix a place $v \dvd p$ in $F^+$ and form the family $D_{\rig}^{\dagger}(\rho_{Y,\twid v})$ of $(\varphi,\Gamma_K)$-modules over $Y$ (see \cite{KedlayaLiu-FamiliesofPhiGammaModules}). Consider the natural map $\chi: Y \goto\hat T$ as the universal character $\chi: T \goto \cal O(Y)^\times$.\footnote{Here and below we are using $\cal O(-)$ to denote the ring of rigid analytic functions on a rigid space.} Write $\chi_v: T(F_v^+) \goto \cal O(Y)^\times$ for the local component at $v$, and we further write $\chi_v  = \chi_{1,v}\tensor \dotsb \tensor \chi_{n,v}$, with each $\chi_{j,v}$ a character of $(F_v^+)^\times$. Set $\Delta_{j,v} := \chi_{1,v}\dotsb \chi_{j,v} : (F_v^+)^\times \goto \cal O(Y)^\times$.

Choose a uniformizer $\varpi_{F_v^+} \in (F_v^+)^\times$ and use $\varpi_{F_v^+}$ to write $({F_v^+})^\times = \varpi_{F_v^+}^{\Z} \times \cal O_{{F_v^+}}^\times$. Let $\eta_{j,v} := \restrict{\Delta_{j,v}}{\cal O_{F_v^+}^\times}$, but then use the same notation to denote what we called $(\eta_{j,v})_{\varpi_K}$ in Section \ref{subsec:algebraic-char-deformations}, i.e. set $\eta_{j,v}(\varpi_K) = 1$. With these notations, \cite{Liu-Triangulations} implies that for $j=1,\dotsc, n$ the space $D_{\crys}^+(\wedge^jD_{\rig}^{\dagger}(\rho_{Y,\twid v})(\eta_{j,v}^{-1}))^{\varphi^{f_K} = \Delta_{j,v}(\varpi_{F_v^+})}$ is a coherent sheaf of generic rank one on $Y$ (see the comments preceding \cite[Theorem 0.3.4]{Liu-Triangulations} and note the consistency with \eqref{eqn:crystalline-eigenvalues}). 

We can say more. By assumption (\ref{subsec:ub}-\ref{ass-generic}), the triangulations $P_{\twid v,\bullet}$ at the point $x$ are all regular generic and thus \cite[Proposition 4.3.5]{Liu-Triangulations} implies that after shrinking $Y$ we may assume that each $D_{\crys}^+(\wedge^jD_{\rig}^{\dagger}(\rho_{Y,\twid v})(\eta_{j,v}^{-1}))^{\varphi^{f_K} = \Delta_{j,v}(\varpi_{F_v^+})}$ is free of rank one, and satisfies base change, on $Y$. In particular, this shows that $\hat \rho_{x,\twid v}$ is a weakly-refined deformation of $\rho_{x,\twid v}$.

We now give the standard argument that the map $f$ is surjective. By compactness of the ring $R_{\rho_{x}}^{\Ref,\min}$, it is enough to see that dense subring ${\cal O}_{X,x}^{\rig} \ci \hat{\cal O}_{X,x}^{\rig}$ is contained in the image of $f$. By construction of $X$, ${\cal O}_{X,x}^{\rig}$ is topologically generated by $\scr H(U^p)^{\sph}$ over $\cal O(\hat T)$ inside $\cal O_{X,x}^{\rig}$ (compare with the proof of \cite[Proposition 7.6.10]{BellaicheChenevier-Book}). The image of the algebra $\scr H(U^p)^{\sph}$ in $\cal O_{X,x}^{\rig}$ is in the image of $f$ by Proposition \ref{proposition:minimal-eigenvarieties-info}(c). 

The image of $\cal O_{\hat T,\chi_x}^{\rig}$ in $\cal O_{X,x}^{\rig}$ is generated by $\set{\Delta_{j,v}, \eta_{j,v}}$ for $v \dvd p$ and $j=1,\dotsc,n$ as above (meaning the values of these characters as function on $Y$). Let $\rho_{x}^{\univ}$ be the universal deformation to $R_{\rho_x}^{\Ref,\min}$. If $I \ci R_{\rho_{x}}^{\Ref,\min}$ is a co-finite length ideal, we apply the definition of being weakly-refined with the choice of uniformizer $\varpi_{F_v^+}$ as above, and deduce that $D_{\crys}^+(\wedge ^j D_{\rig}^{\dagger}(\rho_{x,\twid v}^{\univ}/I)(\eta_{\univ,j,v}^{-1}))^{\varphi^{f_K} = \Delta_{\univ,j,v}(\varpi_{F_{v}^+})}$ is free of rank one over $R_{\rho_{x}}^{\Ref,\min}/I\tensor_{L(x)} (F_v^+)_0$, were $\Delta_{\univ,j,v}$ and $\eta_{\univ,j,v}$ are the universal characters of $(F_v^+)^\times \goto (R_{\rho_x}^{\Ref,\min})^\times$ associated to $\rho_x^{\univ}$. Since $\rho_x^{\univ}$ induces $\hat \rho_x$ under the map $f$, we see that  the corresponding compositions $(F_v^+)^\times \goto (R_{\rho_x}^{\Ref,\min}/I)^\times \goto (\cal O_{X,x}^{\rig}/I\cal O_{X,x}^{\rig})^\times$ are equal to $\Delta_{j,v} \bmod I$ and $\eta_{j,v} \bmod I$, respectively (we are carefully using that  $\eta_{\univ,j,v}$ and $\eta_{j,v}$ are both defined to be trivial on $\varpi_{F_v^+}$). Now we take the limit over $I$ being powers of the maximal ideal in ${\mathcal O}_{X,x}^{\rig}$ and conclude $\set{\Delta_{j,v},\eta_{j,v}}$ are in the image of $f$ by Krull's intersection theorem.
\end{proof}
Now denote by $T_{X,x}$ the tangent space to $X$ at the point $x$. Let $\fr t_{\rho_{x}}^{\Ref,\min}$ denote the tangent space for the ring $R_{\rho_{x}}^{\Ref,\min}$. If $\dim_{L(x)} \fr t_{\rho_{x}}^{\Ref,\min} = g$ then $R_{\rho_{x}}^{\Ref,\min}$ is a quotient of a power series ring $L(x)[[u_1,\dotsc,u_g]]$ in $g$ variables (this goes back to Mazur). 

Since $\rho_x^{\perp} \simeq \rho_x(n-1)$, there is a matrix $A_x \in \GL_n(L(x))$ such that $\rho_x^{\perp}(g) = A_{x} \rho_x(g) A_x^{-1}\chi_{\cycl}^{n-1}(g)$ for all $g \in G_F$. The main theorem of \cite{BellaicheChenevier-Sign}  implies that $A_x$ is symmetric and unique up to scalar because $\rho_x$ is absolutely irreducible. We recall that this allows us to extend the action of $G_F$ on $\ad \rho_x$ to an action of $G_{F^+}$ (see the introduction of \cite{Allen-SelmerGroups}, for example). Specifically, we write $\ad \rho_{x} = \GL(M_n(L(x)))$, and if $c \in G_{F^+}$ denotes the choice of a complex conjugation for $F/F^+$ then we let $c$ act by $B \mapsto - A^{-1}_x\cdot  {}^t \!B \cdot A_x$ for all $B \in M_n(L)$. This is well-defined and completely canonical. Let $H^1_f(\ad\rho_x)^+ = H^1_f(G_{F^+},\ad \rho_x)$ denote the global adjoint  Bloch--Kato Selmer group \cite{BlochKato-TamagawaNumbersOfMotives}.


\begin{lemma}\label{lemma:global-bounded}
The restriction map $\rho_{x} \mapsto (\rho_{x,\twid v})$ induces a natural map  $\fr t_{\rho_{x}}^{\Ref,\min} \goto \bigoplus_{v \dvd p} \fr t_{\rho_{x,\twid v}}^{\Ref}$ and this induces an exact sequence
\begin{equation*}
0 \goto H^1_f(\ad \rho_{x})^+ \goto \fr t_{\rho_{x}}^{\Ref,\min} \goto \bigoplus_{v \dvd p} \fr t_{\rho_{x,\twid v}}^{\Ref}/\fr t_{\rho_{x,\twid v},f}
\end{equation*}
of vector spaces over $L(x)$.
\end{lemma}
\begin{proof}
The fact that the natural map exists is clear. Next, $\Gal(F/F^+)=\set{1,c}$ acts on $H^1(G_F,\ad\rho_{x})$ by the paragraph preceding this lemma, and the inflation-restriction sequence implies that $H^1(G_{F^+},\ad\rho_x) = H^1(G_F,\ad\rho_{x})^{c=1}$. It is an elementary calculation that a deformation $\twid \rho \in H^1(G_F,\ad \rho_x)$ of $\rho_x$ to $L(x)[\varepsilon]$ is conjugate self-dual if and only if $\twid \rho \in H^1(G_{F},\ad\rho_x)^{c=1}$. Thus $H^1_f(\ad \rho_x)^+$ is contained in the kernel of the restriction maps. The reverse inclusion follows from the minimal condition since every deformation $\twid \rho \in \fr t_{\rho_x}^{\Ref,\min}$ has the property that the restriction $\twid \rho_{\twid w}$ to a place $\twid w \ndvd p$ lies in $H^1_f(\ad \rho_{x,\twid w}) = \fr t_{\rho_{x,\twid w},f}$ already.
\end{proof}
We now summarize the situation. We have fixed a classical point $x$ on an eigenvariety corresponding to the choices (\ref{subsec:ub}-\ref{ass-point})-(\ref{subsec:ub}-\ref{ass-generic}), and we restricted to the minimal eigenvariety $X$ containing $x$. By (\ref{subsec:ub}-\ref{ass-distinct}) the point $x$ comes naturally equipped with triangulations $P_{\twid v,\bullet}$ of the local Galois representation $D_{\rig}^{\dagger}(\rho_{x,\twid v})$ at a set of distinguished $p$-adic places $\twid v$ of $F$.

\begin{theorem}\label{theorem:main-theorem-text}
With the above notation and assumptions, for each $p$-adic place $v$ of $F^+$ let $(\sigma_{v,\tau})_{\tau}$ be critical type of the triangulation $P_{\twid v,\bullet}$ at the point $x$. Then
\begin{equation*}
\dim_{L(x)} T_{X,x} \leq \dim_{L(x)} \fr t_{\rho_{x}}^{\Ref,\min} \leq \dim_{L(x)} H^1_f(\ad \rho_{x})^+ + \sum_{v \dvd p} \sum_{\tau:F_v^+\goto L(x)} \ell(\sigma_{v,\tau}) + c(\sigma_{v,\tau}).
\end{equation*}
\end{theorem}
\begin{proof}
The surjection $f: R_{\rho_x}^{\Ref,\min} \surject \hat{\cal O}_{X,x}^{\rig}$ in Proposition \ref{prop:galois-surjection} gives rise to a canonical injection $T_{X,x} \inject \fr t_{\rho_x}^{\Ref,\min}$. This proves the first inequality. For the second inequality, Lemma \ref{lemma:global-bounded} implies that
\begin{equation}\label{eqn:first-bound}
\dim_{L(x)} \fr t_{\rho_x}^{\Ref,\min} \leq \dim_{L(x)} H^1_f(\ad \rho_x)^+ + \sum_{v \dvd p} \dim_{L(x)} \fr t_{\rho_{x,\twid v}}^{\Ref}/\fr t_{\rho_{x,\twid v},f}.
\end{equation}
If $v \dvd p$ then $P_{\twid v,\bullet}$ is regular generic by assumption (\ref{subsec:ub}-\ref{ass-generic}) and so Theorem \ref{theorem:best-upper-bound} implies that 
\begin{equation}\label{eqn:local-bounds}
\dim_{L(x)} \fr t_{\rho_{x,\twid v}}^{\Ref}/\fr t_{\rho_{x,\twid v},f} \leq \sum_{\tau: F_v^+\goto L(x)} \ell(\sigma_{v,\tau}) + c(\sigma_{v,\tau}).
\end{equation}
The second inequality in the statement of the theorem now follows from summing \eqref{eqn:local-bounds} over $v \dvd p$ and inserting it into \eqref{eqn:first-bound}.
\end{proof}

\begin{corollary}\label{corollary:final-corollary}
In the situation of Theorem \ref{theorem:main-theorem-text}, suppose that
\begin{enumerate}
\item $\sigma_{v,\tau}$ is a product of distinct simple transpositions for all $v$ and all $\tau$, and 
\item $H^1_f(\ad \rho_{x})^+ = (0)$.
\end{enumerate}
Then $X$ is smooth at $x$ and the map $f: R_{\rho_x}^{\Ref,\min} \goto \hat{\cal O}_{X,x}^{\rig}$ is an isomorphism.
\end{corollary}
\begin{proof}
By Proposition \ref{prop:permutations}, the assumption (a) implies that for each $v$, 
\begin{equation*}
\sum_{\tau: F_v^+ \goto L(x)} \ell(\sigma_{v,\tau}) + c(\sigma_{v,\tau}) = \sum_{\tau: F_v^+\goto L(x)} n = (F_v^+:\Q_p)\cdot n.
\end{equation*}
Thus Theorem \ref{theorem:main-theorem-text} implies that
\begin{equation*}
\dim_{L(x)} T_{X,x} \leq \dim_{L(x)} \fr t_{\rho_{\pi}}^{\Ref,\min} \leq \sum_{v \dvd p} (F_v^+:\Q_p)\cdot n = (F^+:\Q)\cdot n.
\end{equation*}
Let $g = (F^+:\Q)\cdot n$. Then $R_{\rho_{x}}^{\Ref,\min}$ is a quotient of a power series ring over $L(x)$ in $g$ variables. On the other hand, $\hat{\cal O}_{X,x}^{\rig}$ is equidimensional of dimension $g$ by Proposition \ref{proposition:minimal-eigenvarieties-info}(a). Krull's Hauptidealsatz then implies that the surjection $f: R_{\rho_{\pi}}^{\Ref,\min} \surject \hat{\cal O}_{X,x}^{\rig}$ is an isomorphism and that $R_{\rho_{x}}^{\Ref,\min}$ and $\hat{\cal O}_{X,x}^{\rig}$ are both power series rings in $g$ variables.
\end{proof}
We finish with a number of remarks.
\begin{remark}
The technical hypothesis on the vanishing of the Selmer group in Corollary \ref{corollary:final-corollary} is expected to always hold. The current status is discussed following Corollary \ref{corollary:intro-corollary}.
\end{remark}
\begin{remark}
The irreducibility of the global Galois representation $\rho_x$ plays a  role, even if under-emphasized, in the proof of Theorem \ref{theorem:main-theorem-text}. Indeed, if we only take as input the inequality $\dim_{L(x)}T_{X,x} \leq \dim_{L(x)} H^1_{f}(\ad \rho_x)^+ + \sum_{v\dvd p} \dotsb$ in Theorem \ref{theorem:main-theorem-text}, and the assumptions (a) and (b) in Corollary \ref{corollary:final-corollary}, then we would still be able to prove that $\dim_{L(x)}T_{X,x} = (F^+:\Q)\cdot n$, meaning $X$ would be smooth at $x$. But, in \cite{Bellaiche-Nonsmooth}, Bella\"iche has given an example  of a classical point on an eigenvariety whose critical type is a 3-cycle (hence a product of distinct simple transpositions) and which is a singular point on every irreducible component it lies on. But, $\rho_x$ is a direct sum of two representations in this case.
\end{remark}

\begin{remark}
In \cite{BreuilHellmannSchraen-LocalModel}, Breuil, Hellmann and Schraen have shown that Corollary \ref{corollary:final-corollary} is optimal in the sense that $X$ is {\em singular} at classical points whose critical types are not all products of distinct simple transpositions. (They had previously shown this in \cite[Section 5]{BHS-Classicality} assuming certain modularity conjectures.)
\end{remark}

\begin{remark}
The astute reader may have noticed that the classicality hypothesis in Theorem \ref{theorem:main-theorem-text} may be relaxed. Surely we used classicality to know that the local representations at the $p$-adic places were crystalline. But, we also used it in writing the sum in Theorem \ref{theorem:main-theorem-text} over the $p$-adic places in terms of the {\em critical type}. Even the basic idea of the critical type requires an {\em a priori} reasonable ordering of Hodge--Tate weights. This ordering arises in the classical setting as the corresponding dominant weight. 

But, Theorem \ref{theorem:main-theorem-text} holds also for (sufficiently generic) crystalline points (with globally irreducible Galois representations) on the eigenvarieties, provided we replace the sum $\sum_{v \dvd p} \sum_{\tau} \dotsb$ by a sum related to the Bloch--Kato Selmer dimensions as in Theorem \ref{theorem:HT-ref-bound}, i.e. a dimension depending on the parameter of the corresponding point on the eigenvariety. Doing this, it seems likely that the corresponding bounds will {\em not} always be tight. 

The easiest examples we have in mind are {\em companion points} on the Coleman--Mazur eigencurve. For concreteness, we may consider an overconvergent $p$-adic cuspform $g$ of negative weight $2-k$, such that $\theta^{k-1}(g)$ is the critical $p$-stabilization of a $p$-ordinary CM form of weight $k$. The analog of Theorem \ref{theorem:main-theorem-text} only produces an upper bound of two for the size of the tangent space to the one-dimensional eigencurve at the point corresponding to $g$. 

However, we learned from Bella\"iche (in a preprint which is no longer publicly available) that if one could prove that $g$ lies on a union of CM components then in fact $g$ lies on a unique component and is smooth. We can sketch our own proof as well. If $g$ lies on a union of CM components then the infinitesimal deformations of the Galois representation on the eigencurve would all be locally split at $p$. The locally split condition is certainly not implied by the weakly-refined condition and thus imposing it would bring the bound of two which Theorem \ref{theorem:main-theorem-text} gives down to a bound of one, which proves $g$ is a smooth point. It bears mentioning that we do not know whether or not such forms $g$ must {\em a priori} lie on a union of CM components, but it would be remarkable if that condition was {\em never} satisfied.
\end{remark}

\bibliography{../../../bibliography/master}
\bibliographystyle{abbrv}

\end{document}